\documentclass[10 pt, reqno]{amsart}
\usepackage{fullpage, amsthm, amsmath, amsfonts, amssymb, url, mathrsfs, stmaryrd, verbatim}

\usepackage{pxfonts}

\author{Casey Rodriguez}
\address{Department of Mathematics\\
  University of Chicago\\
  5734 S. University Avenue \\
  Chicago, IL, 60637}
\email{c-rod216@math.uchicago.edu}

\numberwithin{equation}{section}
\DeclareMathOperator{\arcsinh}{arcsinh}
\newcommand{\R}{\mathbb R}

\newcommand{\N}{\mathbb N}
\newcommand{\Z}{\mathbb Z}
\newcommand{\lims}{\varlimsup}

\newcommand{\M}{\mathcal M}
\newcommand{\s}{\mathbb S}

\newcommand{\vphi}{\varphi}
\newcommand{\ra}{\rangle}
\newcommand{\la}{\langle}

\newcommand{\rar}{\rightarrow}
\newcommand{\p}{\partial}
\newcommand{\al}{\alpha}
\newcommand{\de}{\delta}
\newcommand{\e}{\epsilon}
\newcommand{\tht}{\theta}
\newcommand{\lam}{\lambda}
\newcommand{\h}{\mathcal H}

\newcommand{\cl}{\mathcal}

\newcommand{\tk}{\tilde k}

\newtheorem{lem}{Lemma}[section]
\newtheorem{thm}[lem]{Theorem}
\newtheorem{ppn}[lem]{Proposition}
\newtheorem{defn}[lem]{Definition}

\newtheorem{cor}[lem]{Corollary}

\title{Soliton Resolution for Equivariant \\ Wave Maps on a Wormhole: II}

\begin{document}

\begin{abstract}
In this paper, we continue our study of equivariant \emph{wave maps on a wormhole} initiated in \cite{cpr}.  
More precisely, we study finite energy $\ell$--equivariant wave maps from the $(1+3)$--dimensional spacetime $\R \times
(\R \times \s^2) \rar \s^3$ where the metric on $\R \times (\R \times \s^2)$ is given by  
\begin{align*} 
ds^2 = -dt^2 + dr^2 + (r^2 + 1) \left ( d \theta^2 + \sin^2 \theta d \varphi^2 \right ), \quad t,r \in \R, 
(\theta,\varphi) \in \s^2. 
\end{align*}
The constant time slices are each given by a Riemannian manifold $\cl M$ with 
two asymptotically Euclidean ends at $r = \pm \infty$ that
are connected by a 2--sphere at $r = 0$.  The spacetime $\R \times (\R \times \s^2)$ has appeared in the general relativity 
literature as a prototype wormhole geometry (but is not expected to exist in nature).
Each $\ell$--equivariant finite energy wave map can be indexed by its topological degree $n$.  For each $\ell$ and $n$, there exists 
a unique, linearly stable energy minimizing $\ell$--equivariant harmonic map $Q_{\ell,n} : \cl M \rightarrow \s^3$ of degree $n$.  In this work, we prove the soliton 
resolution conjecture for this model.  More precisely, we show that modulo a free 
radiation term every $\ell$--equivariant wave map of degree $n$ converges strongly to $Q_{\ell,n}$. This fully resolves 
a conjecture made by Bizon and Kahl in \cite{biz2}. Our previous 
work \cite{cpr} showed this for the corotational case $\ell = 1$ and established many preliminary results that are used in the 
current work.
\end{abstract}

\maketitle




\section{Introduction}

Fundamental geometric objects which have been intensely studied in recent years
are \emph{wave maps}.  These maps satisfy a \lq manifold valued wave equation' and can be seen as geometric generalizations
of solutions to the standard linear wave equation on flat Minkowski space.  Wave maps are defined as follows.  
Let $(M,g)$ be a $(1+d)$--dimensional Lorentzian spacetime, and let $(N,h)$ be a Riemannian manifold.  A wave map
$U : M \rar N$ is defined to 
be a formal critical point of the action functional 
\begin{align}\label{s01}
\cl S(U,\p U) = \frac{1}{2} \int_M g^{\mu \nu} \la \p_\mu U, \p_\nu U \ra_h dg.
\end{align}
As critical points of this action, a wave map $U$ is satisfies the Euler--Lagrange equations associated to $\cl S$ which in local coordinates are 
\begin{align}\label{s02}
\Box_g U^i + \Gamma^i_{jk}(U)  \p_{\mu} U^j \p_{\nu} U^k g^{\mu \nu} = 0, 
\end{align}
where $\Box_g := \frac{1}{\sqrt{-g}} \p_\mu(g^{\mu \nu} \sqrt{-g} \p_\nu)$ is the D'Alembertian for the 
background metric $g$ and $\Gamma^i_{jk}$ are the Christoffel symbols for the target metric $h$.  The system 
\eqref{s02} is referred to as the \emph{wave map system} or as simply the \emph{wave map equation}.  Note that if 
$(M,g) = (\R^{1+d}, \eta)$, flat Minkowski space, and $N = \R$ then from \eqref{s02} we see that wave maps are simply solutions to the free wave equation 
on $\R^{1+d}$.  

The most studied setting of wave maps has been when $(M,g) = (\R^{1+d}, \eta)$ and $(N,h)$ is a 
$d$--dimensional Riemannian manifold (see the classical reference 
\cite{shat} and the recent review \cite{sch}).  Wave maps are treated as solutions to the initial value problem for \eqref{s02}. 
It is known that solutions starting from small initial data (within a certain smoothness space) are global and behave, 
in a sense, like solutions to the free wave equation on Minkowski space.  Recently, researchers have turned to the problem 
of describing the long time dynamics of generic large data solutions. The guiding principle is the so called 
\emph{soliton resolution conjecture}.  This belief asserts that for most nonlinear dispersive PDE, 
generic globally defined solutions asymptotically decouple into a superposition of nonlinear bulk terms (traveling waves, 
rescaled solitons, etc.) and radiation (a solution to the underlying linear equation).  However, when trying to verify this 
conjecture for wave maps on Minkowski space, complications arise due to the scaling symmetry:
\begin{align*}
U(t,x) \mbox{ solves \eqref{s02}} \implies U_{\lam}(t,x) = U(\lam t, \lam x) \mbox{ solves \eqref{s02}.}
\end{align*}
Due to this scaling symmetry, the long--time dynamics of large data wave maps on $\R^{1+d}$ can be very complex and one 
can have (depending on the geometry of the target and dimension) self--similar solutions, finite time break down via energy concentration, dynamic \lq towers' of solitons, and 
other interesting scenarios.  Therefore, to gain better insight on the role of the soliton resolution 
conjecture it is instructive to consider models when the background metric does not admit such a scaling symmetry.  Moreover, the case of a curved background metric is still relatively unexplored.  These reasons motivated
the following model introduced by Bizon and Kahl \cite{biz2} which we consider in this paper.     

In this work, we continue our study of so called equivariant \emph{wave maps on a wormhole} initiated in \cite{cpr}.  The 
setup is the following.  We consider wave maps $U : \R \times (\R \times \s^2) \rar \s^3$ where the background metric is 
given by 
\begin{align}\label{s02b}
ds^2 = -dt^2 + dr^2 + (r^2 + 1)(d\theta^2 + \sin^2 \tht d\vphi^2 ), \quad t,r \in \R, (\theta,\varphi) \in \s^2.
\end{align}  
Each constant time slice is given by the Riemannian manifold $\cl M := \R \times \s^2$ with metric 
\begin{align}\label{s02c}
ds^2 = dr^2 + (r^2 + 1)(d\theta^2 + \sin^2 \tht d\vphi^2 ), \quad r \in \R, (\theta,\varphi) \in \s^2.
\end{align}
Since $r^2 + 1 \approx r^2$ for large $r$, $\cl M$ has two asymptotically Euclidean ends connected by a 2--sphere 
at $r = 0$ (the throat).  Because of this, the above spacetime has appeared as a prototype `wormhole' geometry in the 
general relativity literature since its introduction by Ellis in the 1970's and popularization by Morris and 
Thorne in the 1980's (see \cite{mt}, \cite{fjtt} and the references therein).  Due to the rotational symmetry of the background and target, it is natural to consider a subclass of 
wave maps $U : \R \times (\R \times \s^2) \rar \s^3$ such that  
\begin{align}\label{s02d}
\exists \ell \in \N, \quad U \circ \rho = \rho^\ell \circ U, \quad \forall \rho \in SO(3).  
\end{align}
Here, the rotation group $SO(3)$ acts on the background and target in the natural way.  The integer 
$\ell$ is commonly referred to as the \emph{equivariance class} and can be thought of as parametrizing a 
fixed amount of angular momentum for the wave map.  If we fix spherical coordinates
$(\psi, \vartheta, \phi)$ on $\s^3$, then from \eqref{s02d} it follows that $U$ is completely determined by the associated 
function $\psi = \psi(t,r)$ and the wave map equation \eqref{s02} reduces to the single scalar semilinear wave equation 
for $\psi$: 
\begin{align}
\begin{split}\label{s04}
&\p_t^2 \psi - \p_r^2 \psi - \frac{2r}{r^2 + 1} \p_r \psi + \frac{\ell(\ell+1)}{2(r^2 + 1)} \sin 2 \psi = 0, \quad (t,r) \in \R \times \R,\\
&\vec \psi(0) = (\psi_0,\psi_1).
\end{split}
\end{align}
Throughout this work we use the notation $\vec \psi(t) = (\psi(t,\cdot), \p_t \psi(t,\cdot))$.  Solutions $\psi$ to \eqref{s04}
will be referred to as $\ell$--equivariant \emph{wave maps on a wormhole}. The equation \eqref{s04} has the following conserved energy along the flow:
\begin{align*}
\mathcal E_\ell(\vec \psi(t)) := \frac{1}{2} \int_\R \left [ |\p_t \psi(t,r)|^2 + |\p_r \psi(t,r)|^2 + 
\frac{\ell(\ell+1)}{r^2 + 1} \sin^2 \psi(t,r)  \right ] (r^2 + 1) dr = \mathcal E_\ell(\vec \psi(0)). 
\end{align*}
In order for the initial data to have finite energy, we must have for some $m,n \in \Z$, 
\begin{align*}
\psi_0(-\infty) = m\pi \quad \mbox{and} \quad \psi_0(\infty) = n\pi. 
\end{align*}
For a finite energy solution $\vec \psi(t)$ to \eqref{s04} to depend continuously on $t$, we must have that $\psi(t,-\infty) = m
\pi$ and $\psi(t,\infty) = n\pi$ for all $t$.  Due to the symmetries $\psi \mapsto m\pi + \psi$ and 
$\psi \mapsto -\psi$ of 
\eqref{s04}, we will, without loss of generality, fix $m = 0$ and assume 
$n \in \N \cup \{0\}$.  Thus, we only consider wave maps which send the left Euclidean end at $r = -\infty$ to the 
north pole of $\s^3$.  The integer $n$ is referred to as the topological degree of the map $\psi$ and, heuristically, 
represents the minimal number of times $\M$ gets wrapped around $\s^3$ by $\psi$.  For each $n \in \N \cup \{0\}$, we denote the 
set of finite energy pairs of degree $n$ by 
\begin{align*}
\mathcal E_{\ell,n} := \left \{ (\psi_0,\psi_1) : \mathcal E_\ell(\psi_0,\psi_1) < \infty, \quad \psi_0(-\infty) = 0, \quad \psi_0(\infty) = n\pi
\right \}.
\end{align*}

As described in \cite{biz2} and in our earlier work \cite{cpr}, there are features of the wave maps on a wormhole equation that make it an attractive model in which to study the soliton resolution conjecture.  The first feature is that we have global well--posedness for arbitrary solutions to \eqref{s04} trivially.  The geometry of the wormhole breaks the scaling invariance that the equation has in the flat case and removes the singularity at the origin.  By a simple contraction argument, conservation of energy and time stepping we easily deduce that every solution to \eqref{s04} is globally defined in time (see Section 3 for more details).  Another feature of \eqref{s04} is that there is an abundance of static solutions to \eqref{s04}.  Such solutions are more commonly referred to as harmonic maps.  More precisely, it can be shown that for every $\ell \in \N$, $n \in \N \cup \{0\}$ there exists a unique solution $Q_{\ell,n}
\in \cl E_{\ell,n}$ to 
\begin{align}\label{s05}
Q'' + \frac{2r}{r^2 + 1} Q' - \frac{\ell(\ell+1)}{2(r^2 + 1)} \sin 2 Q = 0, \quad r \in \R.  
\end{align}
See Section 2 for more details.  

In \cite{biz2} the authors gave mixed numerical and analytic evidence for the following formulation of the soliton resolution conjecture for this model:  for every $\ell \in \N$, $n \in \N \cup \{0\}$, and 
for any $(\psi_0,\psi_1) \in \cl E_{\ell,n}$ there exist a unique global solution $\psi$ to \eqref{s04}
and solutions $\varphi^{\pm}_L$ to the linearized equation 
\begin{align}\label{s06a}
\p_t^2 \varphi - \p_r^2 \varphi - \frac{2r}{r^2 + 1} \p_r \varphi + \frac{\ell(\ell+1)}{r^2 + 1} \varphi = 0, 
\end{align}
such that 
\begin{align*}
\vec \psi(t) = (Q_{\ell,n}, 0) + \vec \varphi^{\pm}_L(t) + o(1), 
\end{align*}
as $t \rar \pm \infty$.  In our earlier work \cite{cpr}, we verified this conjecture in the so called \emph{corotational} case
$\ell = 1$.  In this work, we verify this conjecture for all equivariance classes. 

We note here that a model with features similar to wave maps on a wormhole was previously studied in 
\cite{ls}, \cite{kls1} and \cite{klls2} which served as further motivation and as a road map for the work carried out here.  In these works, the
authors studied $\ell$--equivariant wave maps $U: \R \times (\R \backslash B(0,1)) \rar \s^3$.  In their work, an $\ell$--equivariant 
wave map $U$ is 
determined by the associated azimuth angle $\psi(t,r)$ which satisfies the equation 
\begin{align}
\begin{split}\label{s05a}
&\p_t^2 \psi - \p_r^2 \psi - \frac{2}{r} \p_r \psi + \frac{\ell(\ell+1)}{2(r^2 + 1)}\sin 2 \psi = 0, \quad t \in \
\R, r \geq 1,\\
&\psi(t,1) = 0, \quad \psi(t,\infty) = n \pi, \quad \forall t.
\end{split}
\end{align}
Such wave maps were called $\ell$--equivariant \emph{exterior wave maps}.  Similar to wave maps on a wormhole, global well--posedness
and an abundance of harmonic maps hold for the exterior wave map equation \eqref{s05a}.  In the works 
\cite{ls}, \cite{kls1}, and \cite{klls2}, the authors proved the soliton resolution conjecture for $\ell$--equivariant exterior
wave maps for arbitrary $\ell \geq 1$.  We point out that the geometry of the background $\R \times (\R \backslash B(0,1))$
is still flat and could be considered as artificially removing the scaling symmetry present in the flat case. On the other hand,
the curved geometry of the background considered in this
work is what removes scaling invariance.  This makes wave maps on a wormhole more geometric in nature
while still retaining the properties that make them attractive for studying the soliton resolution conjecture.  However, due to the 
asymptotically Euclidean nature of the wormhole geometry, we are able to adapt techniques developed for the flat case
to this curved geometry. 

We now state our main result.  In what follows we use the following notation.  If $r_0 \geq -\infty$ and 
$w(r)$ is a positive continuous function on $[r_0,\infty)$, then we denote
\begin{align*}
\| (\psi_0,\psi_1) \|_{\h([r_0,\infty); w(r)dr)}^2 :=
\int_{r_0}^\infty \left [ |\psi_0(r)|^2 + |\psi_1(r)|^2 dr \right ] w(r) dr. 
\end{align*}
The Hilbert space $\h([r_0,\infty); w(r)dr)$ is then defined to be the completion of the vector space of $C^\infty_0(r_0,\infty)$ pairs 
with respect to the norm $\| \cdot \|_{\h([r_0,\infty); w(r)dr)}$.  Let $\ell \in \N$ be a fixed equivariance class, and let $n \in \N \cup \{0\}$ be a fixed topological degree. 
In the $n=0$ case, the natural space to place the solution $\vec \psi(t)$ to \eqref{s04} in is the \emph{energy space}
$\h_0 := \h((-\infty,\infty); (r^2 + 1)dr)$.  Indeed, it is easy to show that $\| \vec \psi \|_{\mathcal E_{\ell,0}} \simeq
\| \vec \psi \|_{\h_0}$.  For $n \geq 1$, we measure distance relative to $(Q_{\ell,n},0)$ and define $\h_{\ell,n} := \mathcal E_{\ell,n} - (Q_{\ell,n},0)$
with \lq norm'
\begin{align*}
\| \vec \psi \|_{\h_{\ell,n}} := \| \vec \psi - (Q_{\ell,n},0) \|_{\h_0}.
\end{align*}
Note that $\psi(r) - Q_{\ell,n}(r) \rar 0$ as $r \rar \pm \infty$. The main result of this work is the following.  

\begin{thm}\label{t01}
	For all $(\psi_0,\psi_1) \in \mathcal E_{\ell,n}$, there exists a unique global solution $\vec \psi(t) 
	\in C(\R; \h_{\ell,n})$ to \eqref{s04} which scatters forwards and backwards in time to the harmonic map $(Q_{\ell,n},0)$, i.e. there exist
	solutions $\varphi^{\pm}_L$ to the linearized equation \eqref{s06a} such that  
	\begin{align*}
	\vec \psi(t) = (Q_{\ell,n}, 0) + \vec \vphi_L^{\pm}(t) + o_{\h_0}(1),
	\end{align*}
	as $t \rar \pm \infty$. 
\end{thm}

We now give an outline of the proof and the paper.  The proof is a generalization of that for the corotational case 
$\ell = 1$ in \cite{cpr} and draws from the work \cite{klls2}.  For this model, the set up is as follows.  We first note that
the existence and uniqueness 
of the harmonic map $Q_{\ell,n}$ follows nearly verbatim from the arguments in \cite{cpr} for the special case $\ell = 1$ which are classical ODE type arguments.  This is discussed more in Section 2.  In the remainder of Section 2 we give 
an equivalent reformulation of Theorem \ref{t01} which is simpler to work with.  Instead of studying the azimuth angle $\psi$, we 
study the function $u$ defined by the relation $\psi = Q_{\ell,n} + \la r \ra^\ell u$.  A simple computation shows that $u$ satisfies a radial semilinear wave equation on a higher dimensional wormhole $(\M^d,g)$
\begin{align}
\begin{split}\label{s06}
&\p_t^2 u - \Delta_g u + V(r) u = N(r,u), \quad (t,r) \in \R \times \R, \\
&\vec u(0) = (u_0,u_1) \in \h := \dot H^1 \times L^2(M^d).
\end{split}
\end{align}   
Here, $d = 2\ell + 3$ and the potential $V$ and nonlinearity $N$ are explicit with $V$ arising from linearizing \eqref{s04} 
about the harmonic map $Q_{\ell,n}$.  Our main result, Theorem \ref{t01}, is shown to be equivalent to the statement that every solution 
to \eqref{s06} is global and scatters to free waves on $\M^d$ as $t \rar \infty$ (see Theorem \ref{t21} for the precise 
statement).  The remainder of the work is then devoted to proving Theorem \ref{t21} (the `$u$--formulation' of our main 
result).  In particular, we use the concentration--compactness/rigidity method introduced by Kenig and Merle in there work on the energy--critical Schr\"odinger and wave equations 
\cite{km06} \cite{km08}.  The method has three main steps and is by contradiction.  In the first step, we show that solutions to
\eqref{s06} starting from small initial data scatter to free waves as $t \rar \pm \infty$.  In the second step, we then 
show that if our main result fails, then there exists a nonzero solution $u_*$ to \eqref{s06} which doesn't scatter in 
either direction and is, in a certain sense, minimal.  This minimality imposes the following compactness property on $u_*$: the set
\begin{align*}
K = \{ \vec u_*(t) : t \in \R \}
\end{align*} 
is precompact in $\h$.  These two steps are carried out in Section 3.  We remark here that in the work \cite{klls2} the authors 
established these steps by using delicate estimates and arguments developed in \cite{bulut} and \cite{cpr1} for the energy--critical wave equation on flat space in high dimensions. This is done by using a Strauss estimate to reduce the nonlinearity to an energy--critical 
power on $\R^{1 + (2\ell+3)}$.  However, the arguments we give in this work are much simpler and bypass all of this technical machinery by using only basic Strichartz and Strauss estimates (in fact, our argument also applies to the analogous step in the exterior wave map problem).  In the final and most 
difficult step, we establish the following rigidity result: if $u$ solves \eqref{s06} and 
\begin{align*}
K = \{ \vec u(t) : t \in \R \}
\end{align*}
is precompact in $\h$ then $\vec u = (0,0)$.  This step contradicts the second step and we conclude that our main result 
Theorem \ref{t01} holds.  This is proved in Section 4.  In particular, we show that such a solution $u$ must be a 
static solution to \eqref{s06} which implies $\psi = Q_{\ell,n} + \la r \ra\ell u$ is a harmonic map.  By the 
uniqueness of $Q_{\ell,n}$, it follows that $\vec u = (0,0)$ as desired.  The proof that $u$ must be a static 
solution to \eqref{s06} uses channels of energy arguments rooted in \cite{dkm4} which were then generalized and 
used in the works \cite{kls1} \cite{klls2} on exterior wave maps.  These arguments 
focus only on the behavior of solutions in regions exterior to light cones, and this is what allows us to adapt them to our asymptotically 
Euclidean setting.

\textbf{Acknowledgments}: This work was completed during the author’s doctoral studies at the University of Chicago. The author
would like to thank his adviser, Carlos Kenig, for his invaluable patience, guidance and careful reading of the original manuscript.




\section{Harmonic Maps and a Reduction to Higher Dimensions}

For the remainder of the 
paper we fix an equivariance class $\ell \in \N$, topological degree $n \in \N \cup \{0\}$ and study solutions 
to the wave  map on a wormhole equation 
\begin{align}\label{s21}
\begin{split}
&\p_t^2 \psi - \p_r^2 \psi - \frac{2r}{r^2 + 1} \p_r \psi + \frac{\ell(\ell+1)}{2(r^2 + 1)} \sin 2 \psi = 0, \quad (t,r) \in \R \times \R\\
&\psi(t,-\infty) = 0, \quad \psi(t,\infty) = n \pi, \quad \forall t, \\
&\vec \psi(0) = (\psi_0,\psi_1). 
\end{split}
\end{align}
We recall that the energy
\begin{align}\label{s21e}
\mathcal E_\ell(\psi) = \frac{1}{2} \int \left[ |\p_t \psi|^2 + |\p_r \psi|^2 + \frac{\ell(\ell+1)}{r^2 + 1} 
sin^2 \psi \right] (r^2 + 1)dr
\end{align}
is conserved along and the flow, and so, we take initial data $(\psi_0,\psi_1)$ in the metric space
\begin{align*}
\mathcal E_{\ell,n} = \left \{ 
(\psi_0,\psi_1) : \mathcal E_\ell(\psi_0,\psi_1) < \infty, \quad \psi_0(-\infty) = 0, \quad \psi_0(\infty) = n\pi
\right \}.
\end{align*}
In this section we review the theory of static solutions to \eqref{s21} (i.e. harmonic maps) and reduce the study of $\ell$--equivariant 
wave maps on a wormhole to the study of a semilinear wave equation on a higher dimensional wormhole.  

\subsection{Harmonic Maps}

In this subsection, we briefly review the theory of harmonic maps for \eqref{s21}.  The main result is the following. 

\begin{ppn}\label{pa21}
	There exists a unique solution $Q_{\ell,n} \in \cl E_{\ell,n}$ to the equation 
	\begin{align}\label{sa21}
	Q'' + \frac{2r}{r^2 + 1} Q' - \frac{\ell(\ell+1)}{2(r^2 + 1)} \sin 2 Q = 0.
	\end{align}
	In the case $n = 0$, $Q_{\ell,0} = 0$.  If $n \in \N$, then $Q_{\ell,n}$ is increasing on $\R$, satisfies 
	$Q_{\ell,n}(r) + Q_{\ell,n}(-r) = n\pi$ and there exists $\al_{\ell,n} \in \R$ such that 
	\begin{align*}
	Q_{\ell,n}(r) &= n\pi - \al_{\ell,n} r^{-\ell -1} + O(r^{-\ell - 3}), \quad \mbox{as } r \rar \infty, \\
	Q_{\ell,n}(r) &= \al_{\ell,n} |r|^{-\ell -1} + O(r^{-\ell - 3}), \quad \mbox{as } r \rar -\infty.
	\end{align*}
	The $O(\cdot)$ terms satisfy the natural derivative bounds.  
\end{ppn}

The proof of Proposition \ref{pa21} is nearly identical to the proof of the corresponding statement, 
Proposition 2.1, in \cite{cpr} which was inspired by arguments in \cite{mctr}.  We briefly sketch the argument. 

\begin{proof}[Sketch of Proof]
	We first can use simple ODE arguments to show that every solution $Q$ to \eqref{sa21} is defined on $\R$ 
	and has limits $Q(\pm \infty)$ in $\Z \pi$ or $(\Z + \frac{1}{2}) \pi$. Moreover, if 
	$Q(\pm \infty) \in \Z \pi$, then $Q$ is monotonic and there exist $\al,\beta \in \R$ such that 
	\begin{align}
	\begin{split}\label{sa22}
	Q(r) &= Q(\infty) + \al r^{-\ell - 1} + O(r^{-\ell - 3}), \quad \mbox{as } r \rar \infty, \\
	Q(r) &= Q(-\infty) + \beta r^{-\ell - 1} + O(r^{-\ell - 3}), \quad \mbox{as } r \rar -\infty.
	\end{split}
	\end{align} For existence, we use a classical shooting argument. For $b > 0$, we consider the solution $Q_b$ to 
	\begin{align*}
	&Q_b'' + \frac{2r}{r^2 + 1} Q_b' - \frac{\ell(\ell+1)}{2(r^2 + 1)} \sin 2 Q_b = 0, \\
	&Q_b(0) = \frac{n\pi}{2}, \quad Q_b'(0) = b,
	\end{align*}
	and show the existence of a special value $b_*$ for the shooting parameter $b$ such that 
	$Q_{b_*}(\infty) = n\pi$.  Indeed, using the properties of general solutions to \eqref{sa21} already outlined 
	and simple ODE arguments, we can show that the sets
	\begin{align*}
	B_< &= \{ b > 0 : Q_b(\infty) < n\pi \}, \\ 
	B_> &= \{ b > 0 : Q_b(\infty) > n\pi \},
	\end{align*}
	are both nonempty, open, proper subsets of $(0,\infty)$.  By connectedness, there exists $b_*$ such that $Q_{b_*}(\infty) = 
	n\pi$.  From the initial condition $Q_{b_*} = 0$ and the symmetry $Q(r) \mapsto n\pi - Q(-r)$ of \eqref{sa21}, we conclude that 
	$Q_{b_*}(r) = n\pi - Q_{b_*}(-r)$ as well as $Q_{b_*}(-\infty) = 0$.  We then set $Q_{\ell,n} = Q_{b_*}$.  
	
	For the uniqueness of $Q_{\ell,n}$, suppose that there are two solutions $Q_1, Q_2$ to \eqref{sa21}.  By the previous discussion each solution 
	is monotonic increasing on $\R$ and satisfies \eqref{sa22}.  We change variables to $x = \arcsinh r$ so that
	\eqref{sa21} becomes  
	\begin{align}
	Q'' + \tanh x Q' - \frac{\ell(\ell+1)}{2}\sin 2 Q = 0. \label{sa23}
	\end{align}
	Based on \eqref{sa22} (in the $x$--variable) and \eqref{sa23}, we can then show that if we assume, without loss of generality, that 
	$\frac{dQ_2}{dx} > \frac{dQ_1}{dx}$ for $x$ large and positive, then 
	\begin{align*}
	\frac{dQ_2}{dx} > \frac{dQ_1}{dx}, \quad \forall x \in \R. 
	\end{align*}
	However, this can easily be shown to be incompatible with \eqref{sa22} as $x \rar -\infty$.  Thus, $Q_1 = Q_2$, and the solution 
	$Q_{\ell,n}$ constructed is unique.  For the full details of the argument in the $\ell =1$ case, see Section 2 in \cite{cpr}. 
\end{proof}

A fact that will be essential in the final section of this work is that we may always find a unique solution to \eqref{sa21}
with prescribed asymptotics as either $r \rar \infty$ or $r \rar -\infty$ (but not necessarily both).  

\begin{ppn}\label{pa22}  
	Let $\al \in \R$.  Then there exists a unique solution $Q^+_{\al}$ to \eqref{sa21} such that 
	\begin{align*}
	Q^+_\al = n \pi + \al r^{-\ell - 1} + O( r^{-\ell - 3}), \quad \mbox{as } r \rar \infty. 
	\end{align*}
	Similarly, given $\beta \in \R$, there exists a unique solution $Q^-_{\beta}$ to \eqref{sa21} such that 
	\begin{align*}
	Q^-_\beta = \beta r^{-\ell - 1} + O( r^{-\ell - 3}), \quad \mbox{as } r \rar -\infty. 
	\end{align*}
\end{ppn}

\begin{proof}
	The proof is nearly identical to the proof of Proposition 2.4 in \cite{cpr} and we omit the details. 
\end{proof}

\subsection{Reduction to a Wave Equation on a Higher Dimensional Wormhole}

In this subsection we reduce the study of the large
data solutions to \eqref{s21} to the study of large data solutions to a semilinear wave equation on a 
higher dimensional wormhole geometry. This process is a generalization of the analogous 
step in the corotational case in \cite{cpr}.  

By Proposition \ref{pa21}, there exists a unique static solution $Q_{\ell,n}(r) \in \cl E_{\ell,n}$ to \eqref{s21}.
For a solution $\psi$ to \eqref{s21}, we define $\vphi$ by 
\begin{align*}
\psi(t,r) = Q_{\ell,n}(r) + \vphi(t,r).
\end{align*}
Then \eqref{s21} implies that $\vphi$ satisfies
\begin{align}\label{s23} 
\begin{split}
&\p_t^2 \vphi - \p_r^2 \vphi - \frac{2r}{r^2 + 1} \p_r \vphi + \ell(\ell+1)\frac{\cos 2 Q_{\ell,n}}{r^2 + 1} \vphi = Z(r,\vphi), \\
&\vphi(t,-\infty) = \vphi(t,\infty) = 0, \quad \forall t, \\
&\vec \vphi(0) = (\psi_0-Q_{\ell,n},\psi_1), 
\end{split}
\end{align}
where 
\begin{align*}
Z(r,\phi) = \frac{\ell(\ell+1)}{2(r^2 + 1)} \left [ 2 \vphi - \sin 2\vphi \right ] \cos 2 Q_{\ell,n} + ( 1 - \cos 2 \vphi) 
\sin 2 Q_{\ell,n}.
\end{align*}
The left--hand side of \eqref{s23} has more dispersion than a free wave on $\M^3$ due to the repulsive potential 
\begin{align*}
\ell(\ell+1)\frac{\cos 2 Q}{r^2 + 1} = \frac{\ell(\ell+1)}{r^2 + 1} + O(\la r \ra^{-2\ell-4} )
\end{align*}
as $r \rar \pm \infty$. Here and throughout this work, we use the Japanese bracket notation $\la r \ra = (r^2 + 1)^{1/2}$.  The $O(\cdot)$ term is a consequence of the 
asymptotics from Proposition \ref{pa21}.   We now make a standard
reduction that incorporates this extra dispersion.  We define $u$ and $d$ via the relations 
\begin{align*}
\vphi &= \la r \ra^{\ell} u, \\
d &= 2 \ell + 3.
\end{align*}
We define the $d$--dimensional wormhole $\M^d = \R \times \s^{d-1}$ with metric 
\begin{align*}
ds^2 = dr^2 + (r^2 + 1) d\Omega^2_{d-1}, 
\end{align*}
where $d\Omega^2_{d-1}$ is the standard round metric on $\s^{d-1}$.  Since we will only be dealing with 
functions depending solely on $r$, we will abuse notation slightly and denote 
the radial part of the Laplacian on $\M^d$ by $-\Delta_g$, i.e. 
\begin{align*}
-\Delta_g u = - \p_r^2 u - \frac{(d-1)r}{r^2 + 1} \p_r u.
\end{align*}
By \eqref{s23}, $u$ satisfies
the radial semilinear wave equation
\begin{align}\label{s24} 
\begin{split}
&\p_t^2 u - \Delta_g u + V(r) u = N(r,u), \\
&u(t,-\infty) = u(t,\infty) = 0, \quad \forall t, \\
&\vec u(0) = (u_0,u_1),
\end{split}
\end{align}
where the potential term is given by 
\begin{align}\label{s25}
V(r) = \frac{\ell^2}{\la r \ra^4} + \ell(\ell+1) \frac{ \cos 2 Q - 1 }{\la r \ra^2}, 
\end{align}
and the nonlinearity $N(r,u) = F(r,u) + G(r,u)$ is given by 
\begin{align}
\begin{split}\label{s26}
F(r,u) &= \frac{\ell(\ell+1)}{\la r \ra^{\ell+2}} \sin^2 (\la r \ra^\ell u) \sin 2 Q_{\ell,n}, \\
G(r,u) &= \frac{\ell(\ell+1)}{2 \la r \ra^{\ell+2}} \left [ 2 \la r \ra^\ell u - \sin (2 \la r \ra^\ell u) \right ] \cos 2 Q_{\ell,n}.
\end{split}
\end{align}
By Proposition \ref{pa21}, the potential $V$ is smooth and satisfies
\begin{align}\label{s27}
V(r) = \frac{\ell^2}{\la r \ra^4} + O ( \la r \ra^{-2\ell - 4} ).
\end{align}
Also, by Proposition \ref{pa21} $Q_{\ell,n}(-r) + Q_{\ell,n}(r) = n\pi$ which implies that $V(r)$ is an even function.  The nonlinearities $F$ and $G$ satisfy 
\begin{align}
F(r,u) &= \left ( \ell(\ell+1) \la r \ra^{\ell - 2} \sin 2 Q_{\ell,n}  \right ) u^2 + F_0(r,u), \label{s28a}
\end{align}
where 
\begin{align}
|F_0(r,u)| &\lesssim \la r \ra^{2\ell - 3} |u|^4, \label{s28b}
\end{align}
and 
\begin{align}
|G(r,u)| &\lesssim \la r \ra^{2\ell - 2} |u|^3, \label{s29}
\end{align}
where the implied constants depend only on $\ell$.  Since the original azimuth angle $\psi = Q_{\ell,n} + \la r \ra^\ell u
\in \cl E_{\ell,n}$ 
, we take initial data $(u_0,u_1) 
\in \mathcal H(\R; \la r \ra^{d-1} dr)$ for \eqref{s24}.  For the remainder of this section and the next we denote 
\begin{align*}
\h_0 := \h(\R; \la r \ra^2 dr), \quad \h:= \h(\R; \la r \ra^{d-1} dr),
\end{align*}  
and note that 
$\h_0$ is simply the space of radial functions in $\dot H^1 \times L^2(\M^3)$ while $\h$ is the space 
of radial functions in $\dot H^1 \times L^2(\M^d)$. 

In the remainder of the paper, we work in the \lq $u$--formulation' rather than with the original azimuth angle 
$\psi$.  We first show that a solution $\vec \psi(t) \in C(\R; \cl \h_n)$ to \eqref{s21} with initial data $(\psi_0,\psi_1) \in 
\cl E_{\ell,n}$ yields a solution $\vec u(t) \in C(\R; \h)$ with initial data $(u_0,u_1) = \la r \ra^{-\ell}( \psi_0 - Q_{\ell,n}, \psi_1)
\in \h$ and vice versa.  The only fact that needs to be checked is that 
\begin{align}\label{s210}
\| \vec u \|_{\h} \simeq \left \| \vec \psi - (Q_{\ell,n},0) \right \|_{\h_0}.
\end{align}
We define $\vphi = \psi - Q_{\ell,n} = \la r \ra^\ell u$ and compute 
\begin{align}
\p_r \vphi = \la r \ra^\ell \p_r u + \ell r \la r \ra^{\ell-2} u. \label{s211}
\end{align}
We first note that by the fundamental theorem of calculus, we have the Strauss estimates
\begin{align}
\begin{split}\label{s211s}
|\vphi(r)| &\lesssim \la r \ra^{-1/2} \left ( \int |\p_r \vphi|^2 \la r \ra^2 dr \right )^{1/2}, \\
|u(r)| &\lesssim \la r \ra^{(2-d)/2} \left ( \int |\p_r u|^2 \la r \ra^{d-1} dr \right )^{1/2}.
\end{split}
\end{align}
Using the Strauss estimates and integration by parts, we have the following Hardy's inequalities,
\begin{align}
\int |\vphi|^2 dr &\lesssim \int |\p_r \vphi|^2 \la r \ra^2 dr, \notag \\
\int |u|^2 \la r \ra^{d-3} dr &\lesssim \int |\p_r u|^2 \la r \ra^{d-1} dr. \label{s211h}
\end{align}
Recalling that $d$ and $\ell$ are related by $d = 2 \ell + 3$, we see that the relation \eqref{s211} and the two Hardy's inequalities
immediately imply \eqref{s210}.  Hence, the two Cauchy problems \eqref{s21} and \eqref{s24} are equivalent.  

The equivalent $u$--formulation of our main result, Theorem \ref{t01}, is the following. 

\begin{thm}\label{t21}
	For any initial data $(u_0,u_1) \in \h$, there exists a unique global solution $\vec u(t) \in C(\R; \h)$ to \eqref{s24}
	which scatters to free waves on $\M^d$, i.e. there exist solutions $v_L^{\pm}$ to 
	\begin{align*}
	\p_t^2 v - \p_r^2 v - \frac{(d-1)r}{r^2 + 1} \p_r v = 0, \quad (t,r) \in \R \times \R, 
	\end{align*}
	such that 
	\begin{align*}
	\lim_{t \rightarrow \pm \infty} \| \vec u(t) - \vec v_L^{\pm}(t) \|_{\h} = 0. 
	\end{align*}
\end{thm}

The remainder of this work is devoted to proving Theorem \ref{t21}.




\section{Small Data Theory and Concentration--Compactness}

In this section we begin the proof of Theorem \ref{t21} and the study of the nonlinear evolution introduced in the previous section:
\begin{align}
\begin{split}\label{s31}
&\p_t^2 u - \Delta_g u + V(r) u = N(r,u), \quad (t,r) \in \R \times \R, \\
&\vec u(0) = (u_0,u_1) \in \h,
\end{split}
\end{align}
where $\h := \h(\R; \la r \ra^{d-1} dr)$, $d = 2 \ell + 3$, $-\Delta_g$ is the (radial) Laplace operator on the $d$--dimensional
wormhole $\M^d$, and $V(r)$ and 
$N(r,u)$ are given in \eqref{s25} and \eqref{s26}.  

As summarized in the introduction, the proof of Theorem \ref{t21}, or equivalently 
Theorem \ref{t01}, uses the powerful concentration--compactness/rigidity 
methodology introduced by Kenig and Merle in their study of energy--critical dispersive equations \cite{km06} \cite{km08}.  
This methodology was used in the corototational case, $\ell = 1$, $d = 5$, in \cite{cpr}.  The general situation $\ell \in \N$ 
requires many refinements due to the growing dimension $d$.    

The proof of Theorem \ref{t21} 
is split up into three main steps and is by contradiction.  In the first step, we establish small data global well--posedness 
and scattering for \eqref{s31}.  In particular, we establish Theorem \ref{t21} if  $\| (u_0,u_1) \|_{\h} \ll 1$.  In the 
second step, we use the first step and a concentration--compactness argument to show that the \emph{failure} of Theorem \ref{t21} implies
that that there exists a nonzero \lq critical element' $u_*$;  a minimal non--scattering global solution to \eqref{s31}. 
The minimality of $u_*$ imposes the following compactness property on $u_*$: the trajectory 
\begin{align*}
K = \left \{ \vec u_*(t) : t \in \R \right \} 
\end{align*}
is precompact in $\h$. In the third and final step, we establish  the following rigidity theorem: every solution $u$ with $\{ \vec u(t) : t \in \R \}$
precompact in $\h$ must be identically 0.  This contradicts the second step which implies that Theorem \ref{t21} holds.

In this section 
we complete the first two two steps in the program: small data theory and concentration--compactness. 
The proofs for these steps are straightforward generalizations of or nearly identical to those in the corototational case in \cite{cpr}.  We will therefore only outline the main steps and refer 
the reader to the relevant proofs in \cite{cpr} for full details.

\subsection{Small Data Theory}

In this subsection, we establish global well--posedness and scattering for small data solutions to \eqref{s31}.  The key 
tools for establishing this and facts found later in this section are Strichartz estimates
for the inhomogeneous wave equation with potential
\begin{align}
\begin{split}\label{s32}
&\p_t^2 u - \Delta_g u + V(r) u = h(t,r), \quad (t,r) \in \R \times \R, \\
&\vec u(0) = (u_0,u_1) \in \h.
\end{split}
\end{align}
Here, as in the previous section, 
\begin{align*}
-\Delta_g u = - \p_r^2 u - \frac{(d-1)r}{r^2 + 1} \p_r u,
\end{align*}
and the potential $V$ is given by 
\begin{align*}
V(r) = \frac{\ell^2}{\la r \ra^4} + \ell(\ell+1) \frac{ \cos 2 Q_{\ell,n} - 1 }{\la r \ra^2}, 
\end{align*}
where $Q_{\ell,n}$ is the unique $\ell$--equivariant harmonic map of degree $n$.  The conserved energy 
for the homogeneous problem, $h \equiv 0$ in \eqref{s32}, is given by 
\begin{align*}
\cl E_V(\vec u) = \frac{1}{2} \int_\R \left ( |\p_t u|^2 + |\p_r u|^2 + V(r) |u|^2 \right ) \la r \ra^{d-1} dr. 
\end{align*}
In exactly the same fashion as in the corotational case, it can be shown that the operator $-\Delta_g + V(r)$ defined (densely) on $L^2(\cl M^d) = L^2(\R; \la r \ra^{d-1}dr)$ is a nonnegative self--adjoint operator and 0 is neither an eigenvalue nor a resonance.  Moreover, from this spectral information we conclude $\| \vec u \|_{\h}^2 \simeq 
\cl E_V(\vec u)$ along with the following Strichartz estimates (see 
Section 4 and Section 5 of \cite{cpr} for full details of the arguments).  

We say that a triple $(p,q,\gamma)$ is \emph{admissible} if 
\begin{align*}
p > 2, q \geq 2, \quad \frac{1}{p} + \frac{d}{q} = \frac{d}{2} - \gamma, \quad 
\frac{1}{p} \leq \frac{d-1}{2} \Bigl ( \frac{1}{2} - \frac{1}{q} \Bigr ).
\end{align*}
In the sequel, we use the notation for spacetime norms over $I \times \M^d$ via 
\begin{align*}
\| u \|_{L^p_t L^q_x(I)} := \left ( \int_I \left ( 
\int_\R |u(t,r)|^q \la r \ra^{d-1} dr \right )^{p/q} dt \right )^{1/p}.
\end{align*}

\begin{ppn}\label{p31}
	Let $(p,q,\gamma)$ and $(r,s,\rho)$ be admissible.  Then any solution $u$ to \eqref{s32} satisfies 
	\begin{align*}
	\| |\nabla|^{-\gamma} \nabla u \|_{L^p_t L^q_x(I)} \lesssim \| \vec u(0) \|_{\h} +
	\| |\nabla|^{\rho} h \|_{L^{r'}_t L^{s'}_\rho(I)},
	\end{align*}
	where $r'$ and $s'$ are the conjugates of $r$ and $s$.  
\end{ppn}

Proposition \ref{p31} with $V = 0$ was proved in Section 3 of \cite{cpr}.  Using the spectral information 
for $-\Delta_g + V$ we can then transfer these estimates to the perturbed wave operator $\p_t^2 - \Delta_g + V$.
This is done by first reducing Proposition \ref{p31} to a pair of local energy estimates.  These estimates are then established using the spectral information and a distorted Fourier basis for $-\Delta_g + V$ (the fact that $V$ is even also plays a role in the analysis).  Again, for full details see Section 4 of \cite{cpr}.  

For $I \subseteq \R$, we denote the following spacetime norms
\begin{align*}
\| u \|_{S(I)} &:= \Bigl \| \la r \ra^{(d-5)/3} u \Bigr \|_{L^3_t L^6_x(I)}
+ \| u \|_{L^3_t L^{\frac{3d}{2}}_x(I)}, \\
\| u \|_{W(I)} &:= \| u \|_{L^3_t \dot W^{\frac{1}{2},\frac{6d}{3d-5}}_x(I)}, \\
\| h \|_{N(I)} &:= \| F \|_{L^1_tL^2_x(I) + L^{3/2}_t \dot W^{\frac{1}{2},\frac{6d}{3d+5}}_x(I)}.
\end{align*}
We first use Proposition \ref{p31} to show that for any solution $u$ to \eqref{s32}, we have the estimate 
\begin{align}\label{s33a}
\| u \|_{S(I)} + \| u \|_{W(I)} \lesssim \| \vec u(0) \|_{\h} + \| h \|_{N(I)}.  
\end{align}
Indeed, we have directly from Proposition \ref{p31} 
\begin{align}\label{s33}
\| u \|_{W(I)} \lesssim \| \vec u(0) \|_{\h} + \| h \|_{N(I)}.  
\end{align}
We claim that for all radial $f \in C^\infty_0(\M^d)$, we have
\begin{align}\label{s34a}
\Bigl \| \la r \ra^{(d-5)/3} f \Bigr \|_{L^6_x} + \| f \|_{L^{\frac{3d}{2}}_x} \lesssim \| f \|_{\dot W^{\frac{1}{2},\frac{6d}{3d-5}}_x}
\end{align}
(we recall that the volume element is $\la r \ra^{d-1} dr$).  Define $m_0, m_1 > 1$ by the relations 
\begin{align}
\begin{split}\label{s34}
\frac{1}{2} \cdot \frac{1}{3} + \frac{1}{2} \cdot \frac{1}{m_0} &= \frac{3d - 5}{6d}, \\
\frac{1}{2} \cdot \frac{4}{3d} + \frac{1}{2} \cdot \frac{1}{m_1} &= \frac{3d - 5}{6d},
\end{split}
\end{align}
i.e. $\frac{1}{m_0} = \frac{2d - 5}{3d}$ and $\frac{1}{m_1} = \frac{3d-9}{3d}$. By the fundamental theorem of calculus 
\begin{align*}
|f(r)| &\lesssim \| f \|_{\dot W^{1,m_0}} \la r \ra^{-\frac{2}{3}(d-4)}, \\
|f(r)| &\lesssim \| f \|_{\dot W^{1,m_1}} \la r \ra^{-(d-4)}.   
\end{align*}
Thus, we have the embeddings  
\begin{align}
\begin{split}\label{s35}
\Bigl \| \la r \ra^{\frac{2}{3}(d-4)} f \Bigr \|_{L^\infty_x} \lesssim \| f \|_{\dot W^{1,m_0}}, \\
\Bigl \| \la r \ra^{d-4} f \Bigr \|_{L^\infty_x} \lesssim \| f \|_{\dot W^{1,m_1}}.
\end{split}
\end{align}
From the trivial embedding $L^3_x \hookrightarrow L^3_x$, \eqref{s35}, \eqref{s34} and interpolation we conclude that 
\begin{align*}
\Bigl \| \la r \ra^{(d-4)/3} f \Bigr \|_{L^6_x} \lesssim \| f \|_{\dot W^{\frac{1}{2},\frac{6d}{3d-5}}_x} 
\end{align*}
which implies 
\begin{align*}
\Bigl \| \la r \ra^{(d-5)/3} f \Bigr \|_{L^6_x} \lesssim \| f \|_{\dot W^{\frac{1}{2},\frac{6d}{3d-5}}_x}.
\end{align*}
Similarly, from the trivial embedding $L^{\frac{3d}{4}}_x \hookrightarrow L^{\frac{3d}{4}}_x$, \eqref{s35}, \eqref{s34} and interpolation we conclude that 
\begin{align*}
\Bigl \| \la r \ra^{(d-4)/2} f \Bigr \|_{L^{\frac{3d}{2}}_x} \lesssim \| f \|_{\dot W^{\frac{1}{2},\frac{6d}{3d-5}}_x} 
\end{align*}
which implies 
\begin{align*}
\|  f \|_{L^{\frac{3d}{2}}_x} \lesssim \| f \|_{\dot W^{\frac{1}{2},\frac{6d}{3d-5}}_x}.
\end{align*}
This proves the claim.  In particular, $\| u \|_{S(I)} \lesssim \| u \|_{W(I)}$ which along with \eqref{s33}
proves \eqref{s33a}.  Although it may seem redundant to also use the $S(I)$ norm along with the $W(I)$ norm,  it is essential in later 
concentration--compactness arguments to use the weaker norm $\| \cdot \|_{S(I)}$ rather than $\| \cdot \|_{W(I)}$ to measure errors. 

We now use \eqref{s33} to establish an a priori estimate for solutions to \eqref{s31}.  The case $\ell = 1$, $d= 5$
was covered in \cite{cpr}, so we assume that $d \geq 7$.
By the conservation of energy \eqref{s21e}, the Strauss estimate \eqref{s211s}, and 
Hardy's inequality \eqref{s211h} it is easy to show by a 
contraction mapping/time--stepping argument that given $(u_0,u_1) \in \h$, there exists
a unique global solution $\vec u(t) \in C(\R ; \h) \cap L^\infty(\R, \h)$ to \eqref{s31}.  By the Strichartz estimate \eqref{s33a}, we 
have that if 
$u$ solves \eqref{s31}, then for any $I \subseteq \R$, 
\begin{align}
\| u \|_{S(I)} + \| u \|_{W(I)} &\lesssim \| \vec u(0) \|_{\h} + \| N(\cdot, u) \|_{N(I)} \notag \\  
&\lesssim \| \vec u(0) \|_{\h} + \| F(\cdot, u) \|_{N(I)} + \| G(\cdot, u \|_{N(I)}, \label{s36a}
\end{align}
where the nonlinearities $F,G$ are given by \eqref{s26}.  By \eqref{s29} and the relation $d = 2\ell + 3$, we may estimate
\begin{align}\label{s36}
\| G(\cdot, u \|_{N(I)} \lesssim \Bigl \| \la r \ra^{d - 5} u^3 \Bigr \|_{L^1_t L^2_x(I)} \lesssim \| u \|_{S(I)}^3.   
\end{align}
By \eqref{s28a}, \eqref{s28b} and the Strauss estimate \eqref{s211s} we have that 
\begin{align}
\| F(\cdot, u) \|_{N(I)} &\lesssim \Bigl \| 
\left ( \la r \ra^{\ell - 2} \sin 2 Q_{\ell,n}  \right ) u^2
\Bigr \|_{L^{3/2}_t \dot W^{\frac{1}{2},\frac{6d}{3d+5}}_x(I)} + \| F_0 \|_{L^1_t L^2_x(I)} \notag \\
&\lesssim \Bigl \| 
\left ( \la r \ra^{\ell - 2} \sin 2 Q_{\ell,n}  \right ) u^2
\Bigr \|_{L^{3/2}_t \dot W^{\frac{1}{2},\frac{6d}{3d+5}}_x(I)} + \| \vec u \|_{L^\infty_t \h} \| u \|_{S(I)}^3. \label{s37}
\end{align}
By Proposition \ref{pa21} we have 
\begin{align*}
\la r \ra^{\ell - 2} \sin 2 Q_{\ell,n} = O ( \la r \ra^{-3} ), \\
\frac{d}{dr} \Bigl ( \la r \ra^{\ell - 2} \sin 2 Q_{\ell,n} \Bigr ) = O ( \la r \ra^{-4} ),
\end{align*}
so that 
\begin{align}\label{s38}
\Bigl \| \left ( \la r \ra^{\ell - 2} \sin 2 Q_{\ell,n}  \right ) 
\Bigr \|_{L^d_x \cap \dot W^{\frac{1}{2},d}_x} < \infty
\end{align}
by interpolation.  By the Leibniz rule for 
Sobolev spaces (see \cite{coul} for asymptotically conic manifolds) and \eqref{s38}, we conclude that 
\begin{align*}
\Bigl \| 
\left ( \la r \ra^{\ell - 2} \sin 2 Q_{\ell,n}  \right ) u^2
\Bigr \|_{\dot W^{\frac{1}{2},\frac{6d}{3d+5}}_x} 
&\lesssim 
\Bigl \| \left (  \la r \ra^{\ell - 2} \sin 2 Q_{\ell,n}  \right ) 
\Bigr \|_{\dot W^{\frac{1}{2},d}_x} \| u^2 \|_{L^{\frac{6d}{3d-1}}_x} \\
&\:+ \Bigl \| \left (  \la r \ra^{\ell - 2} \sin 2 Q_{\ell,n}  \right ) 
\Bigr \|_{L^d_x} \| u \|_{L^{\frac{3d}{2}}_x} \| u \|_{\dot W^{\frac{1}{2},\frac{6d}{3d-5}}_x} \\
&\lesssim \| u \|^2_{L^{\frac{12d}{3d-1}}_x} + \| u \|_{L^{\frac{3d}{2}}_x} \| u \|_{\dot W^{\frac{1}{2},\frac{6d}{3d-5}}_x}.
\end{align*}
By H\"older's inequality and the fact that $d \geq 7$,
\begin{align*}
\Bigl ( 
\int_\R |u|^{\frac{12d}{3d-1}} \la r \ra^{d-1}
\Bigr )^{\frac{3d -1}{12d}} &\leq \Bigl ( \int_\R \Bigl | \la r \ra^{\frac{d-5}{3}} 
u \Bigr |^6 \la r \ra^{d-1} dr \Bigr )^{\frac{1}{6}} \Bigl ( \int_\R \la r \ra^{\frac{20d - 4d^2}{d-1}} \la r \ra^{d-1}dr \Bigr )
^{\frac{d-1}{3d-1}} \\
&\lesssim 
\Bigl ( \int_\R \Bigl | \la r \ra^{\frac{d-5}{3}} 
u\Bigr |^6 \la r \ra^{d-1} dr \Bigr )^{\frac{1}{6}}. 
\end{align*}
Thus, 
\begin{align*}
\Bigl \| 
\left (  \la r \ra^{\ell - 2} \sin 2 Q_{\ell,n}  \right ) u^2
\Bigr \|_{\dot W^{\frac{1}{2},\frac{6d}{3d+5}}_x} \lesssim 
\Bigl \| \la r \ra^{\frac{d-5}{3}} u \Bigr \|^2_{L^{6}_x} + \| u \|_{L^{\frac{3d}{2}}_x} \| u \|_{\dot W^{\frac{1}{2},\frac{6d}{3d-5}}_x}
\end{align*}
so that by H\"older's inequality in time 
\begin{align}\label{s39}
\Bigl \| 
\left ( \la r \ra^{\ell - 2} \sin 2 Q_{\ell,n}  \right ) u^2
\Bigr \|_{L^{3/2}_t \dot W^{\frac{1}{2},\frac{6d}{3d+5}}_x(I)} \lesssim 
\| u \|_{S(I)}^2 + \| u \|_{S(I)} \| u \|_{W(I)}. 
\end{align}
Combining \eqref{s39} with \eqref{s37} we obtain
\begin{align}\label{s310}
\| F(\cdot, u) \|_{N(I)} \lesssim \| u \|_{S(I)}^2 + \| u \|_{S(I)} \| u \|_{W(I)} + \| \vec u \|_{L^\infty_t \h} 
\| u \|^3_{S(I)}. 
\end{align}
The estimates \eqref{s36a}, \eqref{s36}, and \eqref{s310} imply the following a priori estimate for $u$: 
\begin{align}\label{s311}
\| u \|_{S(I)} + \| u \|_{W(I)} 
\lesssim \| \vec u(0) \|_{\h} + \| u \|_{S(I)}^2 + \| u \|_{S(I)} \| u \|_{W(I)} + \| \vec u \|_{L^\infty_t \h} 
\| u \|^3_{S(I)} + \| u \|_{S(I)}^3.  
\end{align}
Based on \eqref{s311} and continuity arguments we have the following small data theory and long--time perturbation theory 
for \eqref{s31}.  For full details, see the proofs of Proposition 5.1 and Proposition 5.2 respectively in \cite{cpr}. 

\begin{ppn}\label{p32}
	For every $(u_0,u_1) \in \h$, there exists a unique global solution $u$ to \eqref{s31}
	such that $\vec u(t) \in C(\R; \h) \cap L^\infty(\R; \h)$.  A solution $u$ scatters to a free wave on $\M^d$ as $t \rar \infty$, i.e.
	there exists a solution $v_L$ to 
	\begin{align*}
	\p_t^2 v - \p_r^2 v - \frac{(d-1)r}{r^2 + 1} \p_r v = 0, \quad (t,r) \in \R \times \R, 
	\end{align*}
	such that 
	\begin{align*}
	\lim_{t \rightarrow \infty} \| \vec u(t) - \vec v_L^{\pm}(t) \|_{\h} = 0,
	\end{align*}
	if and only if 
	\begin{align*}
	\| u \|_{S(0,\infty)} < \infty. 
	\end{align*}
	A similar characterization of $u$ scattering to a free wave on $\M^d$ as $t \rar -\infty$ also holds.  Moreover, there 
	exists $\delta > 0$ such that if $\| \vec u(0) \|_{\h} < \delta$, then 
	\begin{align*}
	\| \vec u \|_{L^\infty_t \h} + \| u \|_{S(\R)} + \| u \|_{W(\R)} \lesssim \| \vec u(0) \|_{\h}. 
	\end{align*}
\end{ppn}

\begin{ppn}[Long--time perturbation theory]\label{p33}
	Let $A > 0$.  Then there exists $\e_0 = \e_0(A) > 0$ and $C = C(A) > 0$ such that the following holds.  Let $0 < \e
	< \e_0$, $(u_0,u_1) \in \h$, and $I \subseteq \R$ with $0 \in I$.  Assume that $\vec U(t) \in C(I; \h)$ satisfies on $I$
	\begin{align*}
	\p_t^2 U - \Delta_g U + V U = N(\cdot, U) + e, 
	\end{align*}
	such that 
	\begin{align}
	\sup_{t \in I} \| \vec U(t) \|_{\h} + \| U \|_{S(I)} &\leq A, \notag  \\
	\| \vec U(0) - (u_0,u_1) \|_{\h} + \| e \|_{N(I)} &\leq \e.  \label{s312}
	\end{align}
	Then the unique global solution $u$ to \eqref{s31} with initial data $\vec u(0) = (u_0,u_1)$ satisfies
	\begin{align*}
	\sup_{t \in I} \| \vec u(t) - \vec U(t) \|_{\h} + \| u - U \|_{S(I)} \leq C(A) \e. 
	\end{align*}
\end{ppn}

\subsection{Concentration--Compactness}

In this subsection we complete the second step of the concentration--compactness/rigidity method outlined 
in the beginning of this section.  A crucial tool
used in completing this step is the following linear \emph{profile decomposition} of a bounded sequence 
in $\h$.  

\begin{lem}[Linear Profile Decomposition]\label{l34}
	Let $\{ (u_{0,n}, u_{1,n}) \}_n$ be a bounded sequence in $\h$.  Then 
	after extraction of subsequences and relabeling, there exist a sequence of solutions $\left \{ U_L^j \right \}_{j \geq 1}$ to 
	\eqref{s32} with $h \equiv 0$ which are bounded in $\h$ and a sequence of times  $\{ t_{j,n} \}_n$ for $j \geq 1$ that 
	satisfy the orthogonality condition 
	\begin{align*}
	\forall j \neq k, \quad \lim_{n \rar \infty} |t_{j,n} - t_{k,n}| = \infty, 
	\end{align*}
	such that for all $J \geq 1$, 
	\begin{align*}
	(u_{0,n},u_{1,n}) = \sum_{j = 1}^J \vec U^j_L(-t_{j,n}) + (w^J_{0,n},w^J_{1,n}), 
	\end{align*}
	where the error $w_n^J(t) := S_V(t)(w^J_{0,n},w^J_{1,n})$ satisfies
	\begin{align}
	\lim_{J \rar \infty} \lims_{n \rar \infty} \| w^J_n \|_{L^\infty_t L^r_x(\R) \cap S(\R)} = 0, 
	\quad \forall \: \frac{2d}{d-2} < r < \infty. \label{s313}
	\end{align}
	
	Moreover, we have the following Pythagorean expansion of the energy 
	\begin{align}\label{e314}
	\cl E_V( \vec u_n) = \sum_{j = 1}^J \cl E_V( \vec U^j_L ) + \cl E_V ( \vec w^J_n ) + o(1),
	\end{align}
	as $n \rar \infty$. 
\end{lem}

The proof is exactly the same as the corototational case which follows from the proof of Lemma 3.2 in \cite{ls}.  However, 
we will explain why the error $w_n^J$ satisfies \eqref{s313} since the reasoning is subtle.  The $d = 5$ case is contained in \cite{cpr}, so 
we assume that $d \geq 7$.  The proof from Lemma 3.2 in \cite{ls} shows that we have 
\begin{align}\label{e315}
\lim_{J \rar \infty} \lims_{n \rar \infty} \| w^J_n \|_{L^\infty_t L^r_x(\R)} = 0, \quad \forall \: \frac{2d}{d-2} < r < \infty,
\end{align}
as well as 
\begin{align}\label{e316}
\lims_{J \rar \infty} \lims_{n \rar \infty} \| \vec w^J_n \|_{\h} < \infty. 
\end{align}
We recall that in proving \eqref{s34a}, we in fact proved the stronger claim that  
\begin{align}\label{e317}
\Bigl \| \la r \ra^{(d-4)/2} f \Bigr \|_{L^{\frac{3d}{2}}_x} + \Bigl \| \la r \ra^{(d-4)/3} f \Bigr \|_{L^6_x} \lesssim \| f \|_{\dot W^{\frac{1}{2},\frac{6d}{3d-5}}_x}.
\end{align}
We also observe that the admissible triple $\Bigl ( \frac{1}{2}, 3, \frac{6d}{3d - 5} \Bigr )$ is not sharp if 
$d \geq 7$, i.e. 
\begin{align}\label{e318}
\frac{1}{3} < \frac{d-1}{2} \Bigl ( \frac{1}{2} - \frac{3d - 5}{6d} \Bigr ).
\end{align}
The two observations \eqref{e317} and \eqref{e318} and continuity imply the following.  Let $r > \frac{2d}{d-2}$
and $0 < \theta < 1$, and define a triple $(p,q,\gamma)$ and exponent $s$ by 
\begin{align}
\begin{split}\label{e319}
\gamma &= \frac{1}{\theta} \frac{1}{2}, \\
\frac{1}{p} &= \frac{1}{\theta} \frac{1}{3},  \\
\frac{1}{p} + \frac{d}{q} &= \frac{d}{2} - \gamma, \\
\frac{1}{s} &= \theta \frac{1}{q} + (1 - \theta) \frac{1}{r}.
\end{split} 
\end{align}
Then as long as $r$ is sufficiently large and $\theta$ is sufficiently close to 1, we have that
$(p,q,\gamma)$ is admissible, $s > 1$, and 
\begin{align}\label{e320}
\|  f \|_{L^{\frac{3d}{2}}_x} + \Bigl \| \la r \ra^{(d-5)/3} f \Bigr \|_{L^6_x} \lesssim \| f \|_{\dot W^{\frac{1}{2}, s}_x }, \quad \forall
f \in C^\infty_0. 
\end{align}
Indeed, the fact that $(p,q,\gamma)$ defined by \eqref{e319} are admissible for $r$ large and $\theta$ close to 1 follows from 
\eqref{e318} and continuity in $\theta$.  Similarly, by \eqref{e317} if $r$ is large and $\theta$ is close to 1, then we can find 
$m_0 = m_0(\theta)$ and  $m_1 = m_1(\theta)$, analogous to $m_0, m_1$ from \eqref{s34}, so that 
\begin{align*}
\frac{1}{2} \cdot \frac{1}{3} + \frac{1}{2} \cdot \frac{1}{m_0} &= \frac{1}{s}, \\
\frac{1}{2} \cdot \frac{4}{3d} + \frac{1}{2} \cdot \frac{1}{m_1} &= \frac{1}{s},  
\end{align*}
with 
\begin{align*}
|f(r)| &\lesssim \| f \|_{\dot W^{1,m_0}} \la r \ra^{-\al}, \\
|f(r)| &\lesssim \| f \|_{\dot W^{1,m_1}} \la r \ra^{-\beta},
\end{align*}
where $\al > \frac{2(d-5)}{3}$ and $\beta > 0$.  By interpolation we conclude \eqref{e320}. We now fix $r$ sufficiently large
and $\theta$ sufficiently close to 1 so that if $(p,q,\gamma)$ and $s$ are defined as in \eqref{e319}, then 
$(p,q,\gamma)$ is an admissible triple and \eqref{e320} holds.  Then by \eqref{e320}, interpolation,
and Strichartz estimates we have that the errors satisfy 
\begin{align*}
\|  w^J_n \|_{S(\R)}
&\lesssim \| w^J_n \|_{L^3_t \dot W^{\frac{1}{2}, s}_x (\R)} \\
&\lesssim \| w^J_n \|_{L^p_t \dot W^{\gamma, q}_x(\R)}^\theta \| w^J_n \|_{L^\infty_t L^r_x(\R)}^{1 - \theta} \\
&\lesssim \| w^J_n \|_{\h}^{\theta} \| w^J_n \|_{L^\infty_t L^r_x(\R)}^{1 - \theta} 
\end{align*}
whence by \eqref{e315} and \eqref{e316}
\begin{align*}
\lim_{J \rar \infty} \lims_{n \rar \infty} \| w^J_n \|_{S(\R)} = 0 
\end{align*}
as desired. We remark here that is unclear whether or not the errors satisfy the stronger 
condition $\lim_J \lims_n \| w^J_n \|_{W(\R)} = 0$.  It is for this reason that we used the weaker $S(I)$ norm
in the previous subsection.  

Using Lemma \ref{l34} and Proposition \ref{p33}, we establish that if our main result, Theorem \ref{t21}, fails, then 
there exists a nonzero \lq critical element.'   In particular, we establish the following. 

\begin{ppn}\label{p34}
	Suppose that Theorem \ref{t21} fails.  Then there exists a nonzero global solution $u_*$ to \eqref{s31} such that the 
	set 
	\begin{align*}
	K = \left \{ \vec u_*(t) : t \in \R \right \}
	\end{align*}
	is precompact in $\h$.  
\end{ppn} 

The proof of Proposition \ref{p34} is the same as in the corototational case; see the proof of Proposition 5.3 in \cite{cpr}
for full details.  We remark that proving Proposition \ref{p34} uses the nonlinear perturbation theory, Proposition \ref{p33}, applied to 
the linear profile decompositions provided by Lemma \ref{l34}.  What makes this possible is that the perturbation theory is 
established with certain errors measured in the weaker 
norm $\| \cdot \|_{S(\R)}$ (see \eqref{s312}) and the errors $w^J_n$ in the linear profile decomposition satisfy 
$\lim_J \lims_n \| w^J_n \|_{S(\R)} = 0$ (but possibly not $\lim_J \lims_n \| w^J_n \|_{W(\R)} = 0$).   




\section{Rigidity Theorem}

In this section we prove that the critical element from Proposition \ref{p34} does not exist and conclude the proof of our main 
result Theorem \ref{t21} (equivalently Theorem \ref{t01}).  The main result of this section
is the following. 

\begin{ppn}\label{p51} 
	Let $u$ be a global solution of \eqref{s31} such that the trajectory
	\begin{align*}
	K = \{ \vec u(t) : t \in \R \}
	\end{align*}
	is precompact in $\mathcal H := \h(\R; \la r \ra^{d-1} dr)$.  Then $\vec u = (0,0)$.  
\end{ppn}

We first note that for a solution $u$ as in Proposition \ref{p51}, we have the following uniform control of the energy 
in exterior regions. 

\begin{lem}\label{l52} 
	Let $u$ be as in Proposition \ref{p51}.  Then we have 
	\begin{align}
	\begin{split}\label{s51}
	\forall R \geq 0, \quad \lim_{|t| \rar \infty} \| \vec u(t) \|_{\h(|r| \geq R + |t|; \la r \ra^{d-1} dr )} &=  0, \\
	\lim_{R \rar \infty} \left [ \sup_{t \in \R} \| \vec u(t) \|_{\h(|r| \geq R + |t|; \la r \ra^{d-1} dr)} \right ] &= 0.
	\end{split}
	\end{align}
\end{lem}

To prove that $\vec u = (0,0)$, we proceed as in the corotational case \cite{cpr} and show that $u$ is a finite energy static solution to \eqref{s31}.

\begin{ppn}\label{static soln}
	Let $u$ be as in Proposition \ref{p51}.  Then there exists a static solution $U$ to \eqref{s31} such that $\vec u = (U,0)$.  	
\end{ppn}

We will first show that $\vec u$ is equal to static solutions $(U_{\pm},0)$ on $\pm r > 0$ separately.  The proof for $r < 0$ is identical to the proof for $r > 0$ so we will only consider the case $r > 0$.  The major part of this section is devoted to proving the following.  

\begin{ppn}\label{p53}
	Let $u$ be as in Proposition \ref{p51}.  Then there exists a static solution $(U_+,0)$ such that $\vec u(t,r) = (U_+(r),0)$ for all $t \in \R$ and $r > 0$. 
\end{ppn}


\subsection{Proof of Proposition \ref{p53}} 

Let $\eta > 0$ be arbitrary, and let $u$ be as in Proposition \ref{p51}.  
As in \cite{cpr}, we will show that $\vec u(t,r)$ is equal to a static solution $(U_+(r),0)$ to \eqref{s31} on $\{ t \in \R, r \in (\eta,\infty)\}$.  We now introduce a function related to $u$ that will be integral in the proof.  Define 
\begin{align*}
u_e(t,r) := \frac{\la r \ra^{(d-1)/2}}{r^{(d-1)/2}} u(t,r), \quad (t,r) \in \R \times (0,\infty).
\end{align*}
If $u$ solves \eqref{s31} then $u_e$ solves the following radial semilinear wave equation on $\R^{1+d}$ 
\begin{align}\label{s52e} 
\p_t^2 u_e - \p^2_r u_e - \frac{d-1}{r} \p_r u_e + V_e(r) u_e = N_e (r,u_e), \quad (t,r) \in \R \times (0,\infty),
\end{align}
where 
\begin{align}
V_e(r) = V(r) - \frac{(d-1)(d-4)}{2} r^{-2} \la r \ra^{-2} + \frac{(d-1)(d-5)}{4} r^{-2} \la r \ra^{-4} , \label{s52}
\end{align}
and $N_e(r,u_e) = F_e(r,u_e) + G_e(r,u_e)$ with 
\begin{align}
F_e(r,u_e) &= \frac{\la r \ra^{(d-1)/2}}{r^{(d-1)/2}} F \left (r, \frac{r^{(d-1)/2}}{\la r \ra^{(d-1)/2}} u_e \right ), \label{s53}\\
G_e(r,u_e) &= \frac{\la r \ra^{(d-1)/2}}{r^{(d-1)/2}} G \left (r, \frac{r^{(d-1)/2}}{\la r \ra^{(d-1)/2}} u_e \right ), \label{s54}
\end{align}
where $F$ and $G$ are given in \eqref{s26}.  
Note that for all $R > 0$, we have 
\begin{align}\label{s55}
\| \vec u_e(t) \|_{\h( r \geq R; r^{d-1} dr)} \leq C(R) \| \vec u \|_{\h( r \geq R; \la r \ra^{d-1} dr)},
\end{align}
so that by Lemma \ref{l52}, $u_e$ inherits the compactness properties
\begin{align}
\begin{split}\label{s56}
\forall R > 0, \quad \lim_{|t| \rar \infty} \| \vec u_e(t) \|_{\h( r \geq R + |t|; r^{d-1} dr)} = 0, \\
\lim_{R \rar \infty} \left [ \sup_{t \in \R} \| \vec u_e(t) \|_{\h( r \geq R + |t|; r^{d-1} dr)} \right ] = 0.
\end{split}
\end{align}
We also note that due to \eqref{s27}--\eqref{s29} and the definition of $V_e, F_e,$ and $G_e$, we have for all $r > 0$,
\begin{align} 
| V_e(r) | &\lesssim r^{-4}, \label{s57} \\
|F_e(u_e,r)| &\lesssim r^{-3} |u_e|^{2}, \label{s58} \\
|G_e(u_e,r)| &\lesssim r^{d-5}|u_e|^3, \label{s59}
\end{align} 
where the implied constants depend on $Q_{\ell,n}$ and $d$.

To prove Proposition \ref{p53}, we use channels of energy arguments that  originate in the seminal work \cite{dkm4} on the $3d$ energy--critical wave equation.  These arguments have since been used in the study of equivariant exterior wave maps \cite{kls1} \cite{klls2} and in the proof of the corotational case of Theorem \ref{t01} in \cite{cpr}.  The arguments of this section are
derived from those in \cite{klls2}. 
The proof is split into three main steps.  In the first two steps, we determine the precise asymptotics of 
$$(u_{e,0}(r)
,u_{e,1}(r)) := (u_e(0,r), \p_t u_e(0,r)) \quad \mbox{as } r \rar \infty.$$  
In particular, we show that there exists $\al \in \R$ such that 
\begin{align}
r^{d-2} u_{e,0}(r) &= \al + O(r^{-2}), \label{s510} \\
\int_r^\infty u_{e,1}(\rho) \rho^{2j-1} d\rho &= O(r^{2j - d - 1}), \quad j = 1, \ldots, \left \lfloor \frac{d}{4} \right \rfloor, \label{s511}
\end{align}  
as $r \rar \infty$.  In the final step, we use this information and channels of energy arguments to conclude the proof of Proposition \ref{p53}.  In the remainder of this subsection we denote $\h(r \geq R) := \h (r \geq R; r^{d-1} dr)$.

As in the study of corotational wave maps on a wormhole, the key tool used in  
establishing \eqref{s510} and \eqref{s511} is the following exterior energy estimate for radial free waves 
on Minkowski space $\R^{1+d}$
with $d$ odd.  The case $d = 5$ used for corotational wave maps on a wormhole and exterior wave maps was
proved in \cite{kls1}, and the general case of $d \geq 3$ and odd was proven in \cite{klls1}. 

\begin{ppn}[Theorem 2, \cite{klls1}] \label{p54}
	Let $d \geq 3$ be odd.  Let $v$ be a radial solution to the free wave equation in $\R^{1 + d}$ 
	\begin{align*}
	&\p_t ^2 v - \Delta v = 0, \quad (t,x) \in \R^{1+d}, \\
	&\vec v(0) = (f,g) \in \dot H^1 \times L^2 ( \R^d).
	\end{align*}
	Then for every $R > 0$,
	\begin{align}\label{s512}
	\max_{\pm} \inf_{\pm t \geq 0} \int_{r \geq R + |t|} |\nabla_{t,r} v(t,r)|^2 r^{d-1} dr \geq \frac{1}{2} \| \pi^{\perp}_R (f,g) 
	\|_{\h(r \geq R)},
	\end{align}
	where $\pi_R = I - \pi_R^{\perp}$ is the orthogonal projection onto the plane 
	\begin{align*}
	P(R) = \mbox{span} \Bigl \{ (r^{2i - d},0), (0,r^{2j - d}) : i = 1,\ldots, \Bigl \lfloor \frac{d+2}{4} \Bigr \rfloor, 
	j = 1,\ldots, \Bigl \lfloor \frac{d}{4} \Bigr \rfloor \Bigr \} 
	\end{align*}
	in $\h( r \geq R)$. The left--hand side of \eqref{s512} is identically 0 for data satisfying $(f,g)|_{r \geq R} \in P(R)$
\end{ppn}  

We remark here that Proposition \ref{p54} states, quantitatively, that generic solutions to the free wave equation
on $\R^{1+d}$ with $d$ odd emit 
a fixed amount of energy into regions exterior to light cones.  However, this property
fails in the case $R = 0$ for general data $(f,g)$ in even dimensions (see \cite{cks}).

In the remainder of this subsection, we denote 
\begin{align*}
\tilde k := \left \lfloor \frac{d+2}{4} \right \rfloor, \quad 
k := \left \lfloor \frac{d}{4} \right \rfloor.
\end{align*}
For $R \geq 1$, we define the projection coefficients $\lam_i(t,R),\mu_j(t,R)$ for $i = 1, \ldots, \tilde k$, $j = 1,\dots, k$, via
\begin{align}
\pi_R \vec u_e (t,r) = \left ( u_e(t,r) - \sum_{i = 1}^{\tilde k} \lam_i(t,R) r^{2i-d}, 
\p_t u_e(t,r) - \sum_{j = 1}^{k} \mu_j(t,R) r^{2j-d} \right ). \label{s522} 
\end{align}
We now give identities relating $u_e$ to the coefficients $\lam_j(t,r), \mu_i(t,r)$ and an equivalent way of 
expressing the relative size of $\| \pi_R \vec u_e(t) \|_{\h(r \geq R)}$ and $\| \pi_R^\perp \vec u_e(t) \|_{\h(r \geq R)}$ using projection 
coefficients. 

\begin{lem}[Lemma 4.5, Lemma 5.10, \cite{klls2}]\label{l54} 
	For each fixed $R > 0$, and $(t,r) \in \{ r \geq R + |t| \}$, we have the following identities:
	\begin{align*}
	u_e(t,r) &= \sum_{j = 1}^{\tilde k} \lam_j(t,r) r^{2j-d},  \\
	\int_r^\infty \p_t u_e(t,\rho) \rho^{2i-1} d\rho &= \sum_{j = 1}^k \mu_j(t,r) \frac{r^{2i+2j-d}}{d-2i-2j}, \quad \forall 1 
	\leq i \leq k, \\
	\mu_j(t,r) &= \sum_{i = 1}^k r^{d-2i-2j} \frac{c_i c_j}{d-2i-2j} \int_r^\infty \p_t u_e(t,\rho) 
	\rho^{2i-1} d \rho, \quad \forall 1 \leq i \leq k, \\
	\lam_j(t,r) &= \frac{d_j}{d-2j} \left ( u_e(t,r)r^{d-2j} + 
	\sum_{i = 1}^{\tilde k -1} \frac{(2i)d_{i+1} r^{d-2i-2j}}{d-2i-2j} \int_r^\infty u_e(t,\rho) \rho^{2i-1} d\rho \right ),
	\end{align*}
	where the last identity holds for all $j \leq \tilde k$ and 
	\begin{align*}
	c_j &:= \frac{\Pi_{1 \leq l \leq k} (d - 2j - 2l)}{\Pi_{1 \leq l \leq k, l \neq j} (2l-2j)}, \quad 1 \leq j \leq k, \\
	d_j &:= \frac{\Pi_{1 \leq l \leq \tilde k} (d  + 2- 2j - 2l)}{\Pi_{1 \leq l \leq \tk, l \neq j} (2l-2j)}, \quad 1 \leq j \leq \tilde k.
	\end{align*}
	Also, the following estimates hold
	\begin{align*}
	\| \pi_R \vec u_e(t) \|^2_{\h(r \geq R)} 
	&\simeq 
	\sum_{i = 1}^{\tilde k} \Bigl ( \lam_i(t,R) R^{2i - \frac{d+2}{2}} \Bigr )^2 
	+ \sum_{j = 1}^{k} \Bigl ( \mu_j(t,R) R^{2j - \frac{d}{2}} \Bigr )^2,  \\
	\| \pi^\perp_R \vec u_e(t) \|^2_{\h(r \geq R)} 
	&\simeq \int_R^\infty
	\sum_{i = 1}^{\tilde k} \Bigl ( \p_r \lam_i(t,r) r^{2i - \frac{d+1}{2}} \Bigr )^2 
	+ \sum_{j = 1}^{k} \Bigl ( \p_r \mu_j(t,r) r^{2j - \frac{d-1}{2}} \Bigr )^2 dr,
	\end{align*}
	where the implied constants depend only on $d$. 
\end{lem}

We now proceed to the first step in proving Proposition \ref{p53}. 

\subsubsection*{Step 1: Decay rate for $\pi^{\perp}_R \vec u_e(t)$ in $\h(r \geq R)$}
In this step we establish the following decay estimate for $\pi_R^{\perp} \vec u_e(t)$.  

\begin{lem}\label{l55}
	There exists $R_0 > 1$ such that for all $R \geq R_0$ and for all $t \in \R$ we have 
	\begin{align}\label{s513a}
	\begin{split}
	\| \pi_R^{\perp} \vec u_e(t) \|_{\h(r \geq R)} \lesssim R^{-2} \| \pi_R \vec u_e(t) \|_{\h(r \geq R)} + 
	R^{-d/2} \| \pi_R \vec u_e(t) \|_{\h(r \geq R)}^2 + R^{-1}\| \pi_R \vec u_e(t) \|_{\h(r \geq R)}^3. 
	\end{split}
	\end{align}
\end{lem}


Since we are only interested in the behavior of $\vec u_e(t,r)$ in exterior regions $\{ r \geq R + |t| \}$, we first consider a modified Cauchy problem.  
In particular, we can, by finite speed of propagation, alter $V_e$, $F_e$, and $G_e$ appearing in \eqref{s52e} in the interior region 
$\{ r \leq R + |t| \}$ without affecting the behavior of $\vec u_e$ on the exterior region $\{r \geq R + |t|\}$.

\begin{defn}\label{d56}
	Let $R \geq \eta$.  For a function $f = f(r,u) : [\eta,\infty) \times \R \rar \R$, we define
	\begin{align*}
	f_R(t,r,u) := 
	\begin{cases}
	f(R + |t|, u) \quad &\mbox{if } \eta \leq r \leq R + |t| \\
	f(r,u) \quad &\mbox{if } r \geq R + |t|
	\end{cases}, \quad (t,r,u) \in \R \times [\eta,\infty) \times \R.
	\end{align*}
\end{defn}

We now consider solutions to a modified version of \eqref{s52e}: 
\begin{align}\label{s513}
\begin{split}
&\p_t^2 h - \p_r^2 h - \frac{d-1}{r} \p_r h = N_R (t,r,h), \quad (t, r) \in \R \times (\R \backslash B(0,\eta)), \\
&\vec h(0) = (h_0,h_1) \in \h_0(r \geq \eta), 
\end{split}
\end{align}
where $\h_0(r \geq \eta) = \{ (h_0,h_1) \in \h(r \geq \eta) : h_0(\eta) = 0 \}$ and 
\begin{align*}
N_R(t,r,h) = - V_{e,R}(t,r) h + F_{e,R}(t,r,h) + G_{e,R}(t,r,h). 
\end{align*}
We note that from Definition \ref{d56} and \eqref{s57}, \eqref{s58}, and \eqref{s59}, we have 
\begin{align}
| V_{e,R}(t,r) | &\lesssim 
\begin{cases}\label{s514}
(R + |t|)^{-4} \quad &\mbox{if } \eta \leq r \leq R + |t|, \\
r^{-4} \quad &\mbox{if } r \geq R + |t|, 
\end{cases}\\
|F_{e,R}(t,r,h)| &\lesssim 
\begin{cases}\label{s515} 
(R + |t|)^{-3}|h|^2 \quad &\mbox{if } \eta \leq r \leq R + |t|, \\
r^{-3}|h|^3 \quad &\mbox{if } r \geq R + |t|, 
\end{cases}\\
|G_{e,R}(t,r,h)| &\lesssim 
\begin{cases}\label{s516} 
(R + |t|)^{d-5}|h|^3 \quad &\mbox{if } \eta \leq r \leq R + |t|, \\
r^{d-5}|h|^3 \quad &\mbox{if } r \geq R + |t|.
\end{cases}
\end{align}

\begin{lem}\label{l57}
	There exist $R_0 > 0$ large and $\de_0 > 0$ small such that for all $R \geq R_0$ and all $(h_0,h_1) \in \h_0(r \geq \eta)$ with 
	\begin{align*}
	\| (h_0,h_1) \|_{\h(r \geq \eta)} \leq \de_0, 
	\end{align*}
	there exists a unique globally defined solution $h$ to \eqref{s513} such that 
	\begin{align}\label{s517}
	\left \| r^{(d-4)/3} h \right \|_{L^3_tL^6_x(\R \times (\R^d \backslash B(0,\eta)))}
	\lesssim \| \vec h(0) \|_{\h(r \geq \eta)}.
	\end{align}
	Moreover, if we define $h_L$ to be the solution to the free equation $\p_t^2 h_L - \Delta h_L = 0$, $(t,x) \in \R \times (
	\R^d \backslash B(0,\eta))$, $\vec h_L(0) = (h_0,h_1)$, then 
	\begin{align}\label{s518}
	\begin{split}
	\sup_{t \in \R} \| \vec h(t) - \vec h_L(t) \|_{\h(r \geq \eta)} \lesssim R^{-2} \| \vec h(0) \|_{\h(r \geq \eta)} + 
	R^{-d/2} \| \vec h(0) \|_{\h(r \geq \eta)}^2 + R^{-1} \| \vec h(0) \|_{\h(r \geq \eta)}^3. 
	\end{split}
	\end{align}
\end{lem}

\begin{proof}
	For the proof, we use the shorthand notation $\R^d_* = \R^d \backslash B(0,\eta)$.  
	The small data global well--posedness and spacetime estimate \eqref{s517} follow from standard contraction mapping and continuity 
	arguments using the following Strichartz estimate: 
	if $h$ is a radial solution to  $\p_t^2 h - \Delta h = F$ on $\R \times \R^d_*$ with $h(t,\eta) = 0$, $\forall t$, then 
	\begin{align*}
	\left \| r^{(d-4)/3} h \right \|_{L^3_tL^6_x(\R \times \R^d_*)} 
	\lesssim \| \vec h(0) \|_{\h(r \geq \eta)} + \| F \|_{L^1_t L^2_x(\R \times \R^d_*)}. 
	\end{align*}
	This estimate follows from \cite{hmssz} and an argument similar to the one used to establish \eqref{s33a}.  Rather than 
	give the details for proving \eqref{s517}, we  
	prove \eqref{s518} since the argument is similar.  By the Duhamel formula and Strichartz estimates we have
	\begin{align*}
	\sup_{t \in \R} \| \vec h(t) - \vec h_L(t) \|_{\h(r \geq \eta)} &\lesssim \| N_R(\cdot, \cdot, h) \|_{L^1_t L^2_x(\R \times \R^d_*)} \\
	&\lesssim \| V_{e,R} h \|_{L^1_t L^2_x(\R \times \R^d_*)} + \| F_{e,R}(\cdot, \cdot, h) \|_{L^1_t L^2_x(\R \times \R^d_*)} \\ &\: + 
	\| G_{e,R}(\cdot, \cdot, h) \|_{L^1_t L^2_x(\R \times \R^d_*)}.
	\end{align*}
	The third term is readily estimated by \eqref{s516} and \eqref{s517}
	\begin{align*}
	\| G_{e,R}(\cdot, \cdot, h) \|_{L^1_t L^2_x(\R \times \R^d_*)} 
	\lesssim R^{-1} \| r^{d-4} h^3 \|_{L^1_t L^2_x(\R \times \R^d_*)} 
	\lesssim R^{-1} \| \vec h(0) \|_{\h(r \geq \eta)}^3.  
	\end{align*}
	For the first term we have 
	\begin{align*}
	\| V_{e,R} h \|_{L^1_t L^2_x(\R \times \R^d_*)} \leq 
	\left \|  r^{-(d-4)/3}V_{e,R} \right \|_{L^{3/2}_t L^{3}_x(\R \times \R^d_*)}
	\left \| r^{(d-4)/3} h \right \|_{L^3_t L^6_x(\R \times \R^d_*)}.  
	\end{align*}
	By \eqref{s514} 
	\begin{align*}
	\left \|  r^{-(d-4)/3}V_{e,R} \right \|_{L^{3/2}_t L^{3}_x(\R \times \R^d_*)} \lesssim R^{-2}. 
	\end{align*}
	Thus, by \eqref{s517}
	\begin{align*}
	\left \|  r^{-(d-4)/3}V_{e,R} \right \|_{L^{3/2}_t L^{3}_x(\R \times \R^d_*)}
	\left \| r^{(d-4)/3} h \right \|_{L^3_t L^6_x(\R \times \R^d_*)} \lesssim R^{-2} \| \vec h(0) \|_{\h(r \geq \eta)}.  
	\end{align*}
	Similarly, using \eqref{s515}, \eqref{s517} and the Strauss estimate valid for all radial $f \in C^\infty_0(\R^d_*)$
	\begin{align*}
	|f(r)| \lesssim r^{\frac{2-d}{2}} \| \nabla f \|_{L^2(\R^d_*)},
	\end{align*}
	we conclude that $\| F_{e,R}(\cdot, \cdot,h) \|_{L^1_t L^2_x(\R \times \R^d_*)} \lesssim R^{-d/2} \| h(0) \|^2_{\h(r \geq \eta)}$ which proves \eqref{s518}. 
\end{proof}

\begin{proof}[Proof of Lemma \ref{l55}]
	We first prove Lemma \ref{l55} for $t = 0$.  For $R > \eta$, define the truncated initial data $\vec u_R(0) = (u_{0,R}, 
	u_{1,R}) \in \h_0(r \geq \eta)$ via 
	\begin{align}
	u_{0,R}(r) &= 
	\begin{cases}\label{s519}
	u_e(0,r) \quad &\mbox{if } r \geq R, \\
	\frac{r - \eta}{R-\eta} u_e(0,R) \quad &\mbox{if } r < R,
	\end{cases} \\
	u_{1,R}(r) &= 
	\begin{cases}\label{s520}
	\p_t u_e(0,r) \quad &\mbox{if } r \geq R, \\
	0 \quad &\mbox{if } r < R.
	\end{cases}
	\end{align}
	Note that for $R$ large, 
	\begin{align}\label{s521}
	\| \vec u_R (0) \|_{\h(r \geq \eta)} \lesssim \| \vec u_e(0) \|_{\h(r \geq R)}.
	\end{align}
	In particular, by \eqref{s56} there exists $R_0 \geq 1$ such that for all $R \geq R_0$, $\| \vec u_R(0) \|_{\cl H(r \geq \eta)} \leq \de_0$
	where $\de_0$ is from Lemma \ref{l57}.  Let $u_R(t)$ be the solution to \eqref{s513} with initial data $(u_{0,R}, u_{1,R})$, and let
	$\vec u_{R,L}(t) \in \h_0(r \geq \eta)$ be the solution to the free wave equation $\p_t^2 u_{R,L} - \Delta u_{R,L}
	= 0,$ $(t,x) \in \R \times \R^d_*$, $\vec u_{R,L}(0) = (u_{0,R}, u_{1,R})$. By finite speed of propagation 
	\begin{align*}
	r \geq R + |t| \implies \vec u_R(t,r) = \vec u_e(t,r).
	\end{align*}  
	By Proposition \ref{p54}, for all $t \geq 0$ or for all $t \leq 0$, 
	\begin{align*}
	\| \pi_R^{\perp} \vec u_{R,L}(0) \|_{\h(r \geq R)} \lesssim \| \vec u_{R,L}(t) \|_{\h(r \geq R + |t|)}. 
	\end{align*}
	Suppose, without loss of generality, that the above bound holds for all $t \geq 0$.  By \eqref{s518} we conclude that for all $t \geq 0$
	\begin{align*}
	\| \vec u_e(t) \|_{\h(r \geq R + |t|)} &\geq \| \vec u_{R,L}(t) \|_{\h(r \geq R + |t|)} - \| \vec u_R(t) - \vec u_{R,L}(t) \|_{\h(r \geq \eta)} \\
	&\geq c \| \pi^{\perp}_R \vec u_{R,L}(0) \|_{\h(r \geq R)} - C \Bigl [ R^{-2} \| u_R(0) \|_{\h(r \geq \eta)} + 
	R^{-d/2} \| \vec u_R(0) \|_{\h(r \geq \eta)}^2 + R^{-1} \| \vec u_R(0) \|^3_{\h(r \geq \eta)} \Bigr ]. 
	\end{align*}
	Letting $t \rar \infty$ and using the decay property \eqref{s56} and the definition of $(u_{0,R}, u_{1,R})$, we conclude that 
	\begin{align*}
	\| \pi^{\perp}_R \vec u_e(0) \|_{\h(r \geq R)} \lesssim 
	R^{-2} \| u_e(0) \|_{\h(r \geq R)} + 
	R^{-d/2} \| \vec u_e(0) \|_{\h(r \geq R)}^2 + R^{-1} \| \vec u_e(0) \|^3_{\h(r \geq R)}.
	\end{align*}
	Note that $\| \vec u_e(0) \|_{\h(r \geq R)}^2 = \| \pi^{\perp}_R \vec u_e(0) \|_{\h(r \geq R)}^2 + 
	\| \pi_R \vec u_e(0) \|_{\h(r \geq R)}^2$.  Thus, if we take $R_0$ large enough to absorb terms involving 
	$\| \pi^{\perp}_R \vec u_e(0) \|_{\h(r \geq R)}$ into the left hand side in the previous estimate, we obtain for all
	$R \geq R_0$ 
	\begin{align*}
	\| \pi^{\perp}_R \vec u_e(0) \|_{\h(r \geq R)} \lesssim 
	R^{-2} \| \pi_R u_e(0) \|_{\h(r \geq R)} + 
	R^{-d/2} \| \pi_R \vec u_e(0) \|_{\h(r \geq R)}^2 + R^{-1} \| \pi_R \vec u_e(0) \|^3_{\h(r \geq R)},
	\end{align*}
	as desired.  This proves Lemma \ref{l55} for $t = 0$. 
	
	For general $t = t_0$ in \eqref{s513a}, we first set 
	\begin{align*}
	u_{0,R,t_0} &= 
	\begin{cases}
	u_e(t_0,r) \quad &\mbox{if } r \geq R, \\
	\frac{r - \eta}{R-\eta} u_e(t_0,R) \quad &\mbox{if } r < R,
	\end{cases} \\
	u_{1,R,t_0} &= 
	\begin{cases}
	\p_t u_e(t_0,r) \quad &\mbox{if } r \geq R, \\
	0 \quad &\mbox{if } r < R.
	\end{cases}
	\end{align*}
	By \eqref{s56} 
	we can find $R_0 = R_0(\de_0)$ independent of $t_0$ such that for all 
	$R \geq R_0$
	\begin{align*}
	\| (u_{0,R,t_0}, u_{1,R,t_0}) \|_{\h(r \geq \eta)} \lesssim \| \vec u_e(t_0) \|_{\h(r \geq R)} \lesssim \de_0.
	\end{align*}
	The previous argument for $t_0 = 0$ repeated with obvious modifications yield \eqref{s513a} for $t = t_0$. 
\end{proof}

Before proceeding to the next step, we reformulate the conclusion of Lemma \ref{l55} using the projection coefficients
$\lam_i(t,R), \mu_j(t,R)$ for $\vec u_e(t)$.  
The following is an immediate consequence of Lemma \ref{l54} and Lemma \ref{l55}.

\begin{lem}\label{l58}
	Let $\lam_i(t,R),\mu_j(t,R)$, $1 \leq i \leq \tilde k$, $1 \leq j \leq k$ be the projection coefficients as in \eqref{s522}.  Then there exists $R_0 \geq 1$ such that uniformly in $R > R_0$ and $t \in \R$
	\begin{align*}
	\int_R^\infty
	\sum_{i = 1}^{\tilde k} \Bigl ( \p_r \lam_i(t,r) r^{2i - \frac{d+1}{2}} \Bigr )^2 
	&+ \sum_{j = 1}^{k} \Bigl ( \p_r \mu_j(t,r) r^{2i - \frac{d-1}{2}} \Bigr )^2 dr \\
	&\lesssim  
	\sum_{i = 1}^{\tilde k} R^{4i - d - 6} |\lam_i(t,R)|^2 + R^{8i - 3d - 4} |\lam_i(t,R)|^4
	+ R^{12i - 3d - 8} |\lam_i(t,R)|^6 \\
	&\:+ \sum_{i = 1}^{k} R^{4i - d - 4} |\mu_i(t,R)|^2 + R^{8i - 3d} |\mu_i(t,R)|^4
	+ R^{12i - 3d-2} |\mu_i(t,R)|^6.
	\end{align*} 
\end{lem}


\subsubsection*{Step 2: Asymptotics for $\vec u_e(0)$}  

In this step, we prove that $\vec u_e(0)$ has the asymptotic expansions \eqref{s510}, \eqref{s511} which we now formulate as 
a proposition.


\begin{ppn}\label{p58}  
	Let $u_e$ be a solution to \eqref{s52e} which satisfies \eqref{s56}.   Let $\vec u_e(0) 
	= (u_{e,0},u_{e,1})$.  Then there exists $\alpha \in \R$ such that 
	\begin{align*}
	r^{d-2} u_{e,0}(r) &= \al + O(r^{-2}),  \\
	\int_r^\infty u_{e,1}(\rho) \rho^{2j-1} d\rho &= O(r^{2j - d -1} ), \quad j = 1, \ldots, k, 
	\end{align*}  
	as $r \rar \infty$.  
\end{ppn}

The proof of Proposition \ref{p58} is split up into a several lemmas.  First, we use Lemma \ref{l58} to prove the following difference 
estimate for the projection coefficients. 

\begin{lem}\label{l59}
	Let $\de_1 \leq \de_0$ where $\de_0$ is from Lemma \ref{l57}.  Let $R_1 \geq R_0 > 1$ be large enough so that 
	for all $R \geq R_1$ and for all $t \in \R$
	\begin{align*}
	\| \vec u_e(t) \|_{\h(r \geq R)} &\leq \de_1, \\
	R^{-2} &\leq \de_1. 
	\end{align*}
	Then for all $r, r'$ with $R_1 \leq r \leq r' \leq 2r$ and uniformly in $t$ 
	\begin{align} 
	\begin{split} \label{s528}
	|\lam_j(t,r) - \lam_j(t,r')|
	&\lesssim 
	r^{-2j + 1} \sum_{i = 1}^{\tilde k} r^{2i - 3} |\lam_i(t,r)| + r^{4i - d - 2} |\lam_i(t,r)|^2
	+ r^{6i - d - 4} |\lam_i(t,r)|^3 \\
	&\:+ r^{-2j+1} \sum_{i = 1}^{k} r^{2i - 2} |\mu_i(t,r)|^2 + r^{4i - d} |\mu_i(t,r)|^2
	+ r^{6i - d-1} |\mu_i(t,r)|^3,
	\end{split}
	\end{align}
	and
	\begin{align}
	\begin{split} \label{s529}
	|\mu_j(t,r) - \mu_j(t,r')| &\lesssim 
	r^{-2j} \sum_{i = 1}^{\tilde k} r^{2i - 3} |\lam_i(t,r)| + r^{4i - d - 2} |\lam_i(t,r)|^2
	+ r^{6i - d - 4} |\lam_i(t,r)|^3 \\
	&\:+ r^{-2j} \sum_{i = 1}^{k} r^{2i - 2} |\mu_i(t,r)|^2 + r^{4i - d} |\mu_i(t,r)|^2
	+ r^{6i - d-1} |\mu_i(t,r)|^3.
	\end{split}
	\end{align}
\end{lem}

\begin{proof}
	By the fundamental theorem of calculus and Lemma \ref{l58} we have, for all $r, r'$ such that $R_1 \leq r \leq r' \leq 2r$, 
	\begin{align*}
	|\lam_j(t,r) - \lam_j(t,r')|^2 &=
	\left ( \int_r^{r'} \p_\rho \lam_j(t,\rho) d\rho \right )^2 \\
	&\leq \left ( \int_r^{r'} \rho^{-4j + d +1} d\rho \right ) \left ( \int_r^{r'} \left ( \rho^{2j-\frac{d+1}{2}} \p_{\rho} \lam(t,\rho) \right )^2 d\rho \right ) \\
	&\lesssim r^{-4j+d+2} \sum_{i = 1}^{\tilde k} r^{4i - d - 6} |\lam_i(t,r)|^2 + r^{8i - 3d - 4} |\lam_i(t,r)|^4
	+ r^{12i - 3d - 8} |\lam_i(t,r)|^6 \\
	&\:+ r^{-4j+d+2} \sum_{i = 1}^{k} r^{4i - d - 4} |\mu_i(t,r)|^2 + r^{8i - 3d} |\mu_i(t,r)|^4
	+ r^{12i - 3d-2} |\mu_i(t,r)|^6
	\end{align*}
	which proves \eqref{s528}. 
	
	Similarly, we have 
	\begin{align*}
	|\mu_j(t,r) - \mu_j(t,r')|^2
	&\leq r^{-4j+d} \left ( \int_r^{r'} ( \rho^{2j - \frac{d-1}{2}} \p_{\rho} \mu_j(t,\rho) )^2 d\rho \right ) \\
	&\lesssim r^{-4j+d} \sum_{i = 1}^{\tilde k} r^{4i - d - 6} |\lam_i(t,r)|^2 + r^{8i - 3d - 4} |\lam_i(t,r)|^4
	+ r^{12i - 3d - 8} |\lam_i(t,r)|^6 \\
	&\:+ r^{-4j+d} \sum_{i = 1}^{k} r^{4i - d - 4} |\mu_i(t,r)|^2 + r^{8i - 3d} |\mu_i(t,r)|^4
	+ r^{12i - 3d-2} |\mu_i(t,r)|^6
	\end{align*}
	which proves \eqref{s529}. 
\end{proof}

We note that with $\delta_1$ and $R_1$ fixed as in Lemma \ref{l59}, we have by Lemma \ref{l54} for all $r \geq R_1$ 
and uniformly in time 
\begin{align}
\begin{split}\label{s527}
|\lam_i(t,r)| &\lesssim \de_1 r^{\frac{d+2}{2}-2i}, \quad \forall 1 \leq i \leq \tilde k, \\
|\mu_j(t,r)| &\lesssim \de_1 r^{\frac{d}{2} - 2j}, \quad \forall 1 \leq j \leq k. 
\end{split}
\end{align}
By using this observation, a simple consequence of Lemma \ref{l59} is the following.  

\begin{cor}\label{c510}
	Let $\delta_1$ and $R_1$ be as in Lemma \ref{l59}.  Then for all $r,r'$ with $R_1 \leq r \leq r' \leq 2r$ and for all $t \in \R$ 
	\begin{align}
	|\lam_j(t,r) - \lam_j(t,r')| &\lesssim \de_1 \left ( \sum_{i = 1}^{\tilde k} r^{2i-2j} |\lam_i(t,r)| +
	\sum_{i = 1}^k r^{2i-2j+1} |\mu_i(t,r)| \right ), \label{s530} \\
	|\mu_j(t,r) - \mu_j(t,r')| &\lesssim r^{-1} \de_1 \left ( \sum_{i = 1}^{\tilde k} r^{2i-2j} |\lam_i(t,r)| +
	\sum_{i = 1}^k r^{2i-2j+1} |\mu_i(t,r)| \right ). \label{s531}
	\end{align}
\end{cor}

Before proceeding further, we state point wise and averaged difference (in time) estimates for the projection coefficients that will 
be used in the sequel.  

\begin{lem}[Lemma 5.10, Lemma 5.12 \cite{klls2}]\label{l511a}
	For each $R > 0$, $r \geq R$, and $t_1 \neq t_2$ with $(t_j,r) \in \{ r \geq R + |t| \}$ we have for 
	any $1 \leq j \leq j' \leq \tilde k$
	\begin{align}
	\begin{split}\label{s532a}
	|\lam_j(t_1,r) &- \lam_j(t_2,r)| \\ 
	&\lesssim r^{2j' - 2j} |\lam_{j'}(t_1,r) - \lam_{j'}(t_2,r)| + \sum_{m = 1}^k \left | r^{2m-2j} \int_{t_2}^{t_1} 
	\mu_m(t,r) dt \right |,
	\end{split}
	\end{align}
	as well as for any $1 \leq j \leq k$ 
	\begin{align}\label{s532b}
	\frac{1}{R} \int_R^{2R} \mu_j(t_1,r) - \mu_j(t_2,r) dr = \sum_{i = 1}^{\tilde k} \frac{c_ic_j}{d-2-2j} \int_{t_1}^{t_2} 
	I(i,j) + II(i,j) dt,
	\end{align}
	with 
	\begin{align}
	\begin{split}\label{s532c}
	I(i,j) &= -\frac{1}{R} (u_e(t,r) r^{d-2j-1}) \big |_{r = R}^{r = 2R} + (2i - 2j - 1) \frac{1}{R} \int_R^{2R} u_e(t,r) r^{d- 2j -2} dr \\
	&\:- \frac{(2\ell - 2i + 3)(2i-2)}{R} \int_R^{2R} r^{d-2i-2j} \int_r^\infty u_e(t,\rho) \rho^{2i-3}d\rho dr, \\
	II(i,j) &= \frac{1}{R} \int_R^{2R} r^{d-2i-2j} \int_r^\infty \bigl[ -V_e(\rho) u_e(t,\rho) + N_e(\rho, u_e(t,\rho))
	\bigr ]\rho^{2i-1} d\rho
	dr.  
	\end{split}
	\end{align}
	
\end{lem}

Subtleties arise depending on when $d = 7, 11, 15,\ldots$ ($\ell$ is even) and when 
$d = 5, 9, 13, \ldots$ ($\ell$ is odd) which are due to the relationship between $\tilde k$ and $k$.  We first prove Proposition 
\ref{p58} in the case that \textbf{$\ell$ is even}.  When $\ell$ is even, we have the relations
\begin{align*}
d = 4 \tilde k - 1, \quad \tilde k = k+1.
\end{align*}

We now establish a growth estimate which improves \eqref{s527}. 

\begin{lem}\label{l511}
	Let $\epsilon > 0$ be fixed and sufficiently small.  Then as long as $\delta_1$ as in Lemma 
	\ref{l59} is sufficiently small, we have uniformly in $t$, 
	\begin{align}
	\begin{split}\label{s533}
	|\lam_{\tilde k}(t,r)| &\lesssim r^{\e}, \\
	|\mu_k(t,r)| &\lesssim r^{\e}, \\
	|\lam_i(t,r)| &\lesssim r^{2\tilde k - 2j - 2 + 3\e}, \quad \forall 1 \leq i < \tilde k, \\
	|\mu_i(t,r)| &\lesssim r^{2\tilde k - 2j - 3 + 3\e}, \quad \forall 1 \leq i < k. 
	\end{split}
	\end{align}
\end{lem}

\begin{proof}
	If $r > R_1$, by Corollary 
	\ref{c510} we have, 
	\begin{align}
	|\lam_j(t,2r)| &\leq  (1 + C \de_1) |\lam_j(t,r)| + 
	C \de_1 \left ( \sum_{i = 1}^{\tilde k} r^{2i-2j} |\lam_i(t,r)| +
	\sum_{i = 1}^k r^{2i-2j+1} |\mu_i(t,r)| \right ), \label{s534} \\
	|\mu_j(t,2r)| &\leq  (1 + C \de_1) |\mu_j(t,r)| + 
	r^{-1} C \de_1 \left ( \sum_{i = 1}^{\tilde k} r^{2i-2j} |\lam_i(t,r)| +
	\sum_{i = 1}^k r^{2i-2j+1} |\mu_i(t,r)| \right ). \label{s535} 
	\end{align}
	Fix $r_0 > R_1$ and define 
	\begin{align*}
	a_n :=   \sum_{i = 1}^{\tilde k} (2^n r_0)^{2i-2\tilde k} |\lam_i(t,2^n r_0)| +
	\sum_{i = 1}^k (2^n r_0)^{2i-2\tilde k+1} |\mu_i(t,2^n r_0)| 
	\end{align*}
	then \eqref{s534} and \eqref{s535} imply
	\begin{align*}
	a_{n+1} \leq ( 1 + C(k + \tilde k) \de_1 ) a_n.
	\end{align*} 
	By induction
	\begin{align*}
	a_n \leq ( 1 + C(k + \tilde k) \de_1 )^n a_0 
	\end{align*}
	Choose $\de_1$ so small so that $1 + C(k + \tilde k) \de_1 < 2^{\e}$.  We conclude (using the compactness of 
	$\vec u_e$) that 
	\begin{align*}
	a_n \leq 2^{n \e} a_0 \lesssim 2^{n \e},
	\end{align*}
	whence by our definition of $a_n$
	\begin{align}\label{s536a}
	|\lam_i(t,2^n r_0)| \lesssim (2^n r_0)^{2\tilde k - 2i + \e}, \quad |\mu_i(t,2^n r_0)| \lesssim (2^n r_0)^{2\tilde k - 2i -1 + \e},
	\end{align}
	which is an improvement of \eqref{s527}. 
	
	We now insert \eqref{s536a} back into our difference estimates \eqref{s528} and \eqref{s529}.  We first note that by \eqref{s536a} and 
	the relation $d = 4 \tilde k - 1 \geq 7$, we have the 
	estimates 
	\begin{align}
	\begin{split}\label{537}
	(2^n r_0)^{2i - 3} |\lam_i(t,2^n r_0)| &\lesssim (2^n r_0)^{2\tilde k - 3 + \e}, \\
	(2^n r_0)^{4i - d - 2} |\lam_i(t,2^n r_0)|^2 &\lesssim (2^n r_0)^{-1 + 2\e}
	\lesssim (2^n r_0)^{2\tilde k - 3 + 3\e}, \\
	(2^n r_0)^{6i - d - 4} |\lam_i(t,2^n r_0)|^3 &\lesssim (2^n r_0)^{2\tilde k - 3 + 3\e},
	\end{split}
	\end{align}
	as well as 
	\begin{align} 
	\begin{split}\label{538}
	(2^n r_0)^{2i - 2} |\mu_i(t,2^n r_0)| &\lesssim (2^n r_0)^{2\tilde k - 3 + \e}, \\
	(2^n r_0)^{4i - d} |\mu_i(t,2^n r_0)|^2 &\lesssim (2^n r_0)^{-1 + 2\e}
	\lesssim (2^n r_0)^{2\tilde k - 3 + 3\e}, \\
	(2^n r_0)^{6i - d} |\mu_i(t,2^n r_0)|^3 &\lesssim (2^n r_0)^{2\tilde k - 3 + 3\e}.
	\end{split}
	\end{align}
	Thus, by \eqref{s528} and \eqref{s529}, we deduce that 
	\begin{align}\label{s540}
	|\lam_j(t,2^{n+1} r_0) - \lam_j(t,2^n r_0)| \leq C \de_1 |\lam_j(2^n r_0)| + C (2^n r_0)^{2 \tilde k - 2j - 2 + 3 \e}, \\
	|\mu_j(t,2^{n+1} r_0)- \mu_j(t,2^n r_0)| \leq C \de_1 |\mu_j(2^n r_0)| + C (2^n r_0)^{2 \tilde k - 2j - 3 + 3 \e}.
	\end{align}
	From this we obtain 
	\begin{align*}
	|\lam_j(t,2^{n+1} r_0)| \leq (1+ C \de_1) |\lam_j(2^n r_0)| + C (2^n r_0)^{2 \tilde k - 2j - 2 + 3 \e}. 
	\end{align*}
	Using that we have chosen $\de_1$ so that $(1 + C \de_1) < 2^{\e}$ and iterating we obtain 
	\begin{align*}
	|\lam_j(t,2^n r_0)| \leq (2^\e)^n |\lam_j(t, r_0)| + C \sum_{m = 1}^n 
	(2^m r_0)^{2\tilde k - 2j - 2 + 3 \e} (2^\e)^{n-m}. 
	\end{align*}
	In the case $j = \tilde k$, the previous estimate is easily seen to be $O\left (2^{\e n} \right )$ since the first term dominates, while if $j < \tilde k$, we have the 
	previous estimate is $O\left ((2^n r_0)^{2\tilde k - 2j -2 + 3 \e} \right )$ since then the second term dominates.  A similar 
	argument applies to the $\mu_j$'s, and we conclude that 
	\begin{align}
	\begin{split}\label{s539}
	|\lam_{\tilde k}(t, 2^n r_0)| &\lesssim (2^n r_0)^{\e}, \\
	|\mu_{k}(t, 2^n r_0)| &\lesssim (2^n r_0)^{\e}, \\
	|\lam_{i}(t, 2^n r_0)| &\lesssim (2^n r_0)^{2\tilde k - 2i - 2 + 3\e}, \quad \forall 1 \leq i < \tilde k, \\
	|\mu_{i}(t, 2^n r_0)| &\lesssim (2^n r_0)^{2\tilde k - 2i -3 +  3\e}, \quad \forall 1 \leq i < k.
	\end{split}
	\end{align}
	The estimate \eqref{s539} is uniform in time and an improvement of \eqref{s536a}.  Let $r \geq r_0$ with $2^n r_0 \leq r \leq 2^{n+1}r_0$.  We plug \eqref{s539} into the difference estimate \eqref{s528} and obtain 
	\begin{align*}
	|\lam_{\tilde k}(t,r)| \leq (1 + C \de_1) |\lam_{\tilde k}(t,2^n r_0)| + C (2^n r_0)^{2 \tilde k - 2j - 2 + 3 \e}
	\lesssim (2^n r_0)^{\e} \lesssim r^{\e}. 
	\end{align*}
	The other estimates in \eqref{s533} are obtained by similar reasoning.  This concludes the proof. 
\end{proof}

The following corollary is a consequence of the proof of Lemma \ref{l511}.  

\begin{cor}\label{cor511}
	Let $\e$ and $\de_1$ be as in Lemma \ref{l511},  let $r_0 > R_1$ be fixed, and let $j \in \{ 1, \ldots, \tilde k\}$.  If there exists $a \geq \e$ such that for all $n \in \N$, 
	\begin{align*}
	|\lam_j(t,2^{n+1} r_0)| \leq (1 + C \de_1) |\lam_j(t,2^n r_0)| + (2^n r_0)^a, 
	\end{align*}
	then for all $r \geq r_0$
	\begin{align*}
	|\lam_j(t,r)| \lesssim r^{a}, 
	\end{align*}
	uniformly in time.  A similar statement holds for the $\mu_j$'s as well. 
\end{cor}

We now use the previous lemma as the base case for an induction argument.  The main goal is to prove
the following decay estimates for the projection coefficients.  

\begin{ppn}\label{p512}
	Suppose $d = 7,11,15,\ldots$ and $\e$, $\de_1$,$r_0$ are as in Lemma \ref{l511}.  Then uniformly in time, the following estimates hold:
	\begin{align}
	\begin{split}\label{s541}
	|\lam_j(t,r)| &\lesssim r^{-2j + 3 \e}, \quad \forall 1 < j \leq \tilde k, \\
	|\lam_1(t,r)| &\lesssim r^{\e}, \\
	|\mu_j(t,r)| &\lesssim r^{-2j -1 + 3 \e}, \quad \forall 1 \leq j \leq k.
	\end{split}
	\end{align}
\end{ppn}

Proposition \ref{p512} is a consequence of the following proposition with $P = k$.

\begin{ppn}\label{p513}
	With the same hypotheses as in Proposition \ref{p512}, for $P = 0, 1, \ldots, k$ the following estimates hold uniformly in time:
	\begin{align}
	\begin{split}\label{s542}
	|\lam_j(t,r)| &\lesssim r^{2(\tilde k - P -j) - 2 + 3 \e}, \quad \forall 1 \leq j \leq \tilde k \mbox{ with } j \neq \tilde k - P, \\
	|\lam_{\tilde k - P}(t,r)| &\lesssim r^{\e},\\
	|\mu_j(t,r)| &\lesssim r^{2(k - P -j) - 1 + 3 \e}, \quad \forall 1 \leq j \leq k \mbox{ with } j \neq k - P, \\
	|\mu_{k - P}(t,r) &\lesssim r^{\e}.
	\end{split}
	\end{align}
\end{ppn}

\begin{proof}[Proof of Proposition \ref{p513}]
	As was mentioned before, we prove Proposition \ref{p513} by induction.  The base case $P = 0$ is contained in Lemma 
	\ref{l511}.  We now assume that the estimates \eqref{s542} hold for $P$ with $1 \leq P \leq k - 1$ and wish to show that 
	the estimates \eqref{s542} also hold for $P+1$.  The proof is divided into several lemmas.  The bulk of the argument is devoted to proving that  
	the coefficients $\lam_{\tilde k - P}$ and $\mu_{k - P}$ satisfy certain decay estimates.  We first show that they have spatial limits. 
	
	\begin{lem}\label{l514}
		There exist bounded functions $\al_{\tilde k - P}(t)$ and $\beta_{k - P}(t)$ such that 
		\begin{align}
		|\lam_{\tilde k - P}(t,r) - \al_{\tilde k - P}(t) | = O(r^{-2}), \label{s543} \\
		|\mu_{k - P}(t,r) - \beta_{k - P}(t) | = O(r^{-1}),  \label{s544}
		\end{align}
		where the $O(\cdot)$ terms are uniform in time. 
	\end{lem} 
	
	\begin{proof}
		Fix $r_0 > R_1$.  We insert the estimates \eqref{s542} furnished by our induction hypothesis into the difference estimate \eqref{s528}.  We first note
		that based on \eqref{s542}, we can estimate the sum excluding the coefficients $\lam_{\tilde k - P}$ and $\mu_{k - P}$: 
		\begin{align*}
		\sum_{i \neq \tilde k - P}^{\tilde k}& (2^n r_0)^{2i - 3} |\lam_i(t,2^n r_0)| + (2^n r_0)^{4i - d - 2} |\lam_i(t,2^n r_0)|^2
		+ (2^n r_0)^{6i - d - 4} |\lam_i(t,2^n r_0))|^3 \\
		&+ \sum_{i \neq k - P}^{k} (2^n r_0)^{2i - 2} |\mu_i(t,2^n r_0)|^2 + (2^n r_0)^{4i - d} |\mu_i(t,2^n r_0)|^2
		+ (2^n r_0)^{6i - d-1} |\mu_i(t,2^n r_0)|^3 \\
		&\lesssim 
		(2^n r_0)^{2(\tilde k - P - 1) - 3 + 3\e} + (2^n r_0)^{-4P - 5 + 6 \e} 
		+ (2^n r_0)^{2(\tilde k - 3P - 1) - 7 + 9 \e}. 
		\end{align*}
		In particular, we have the following estimate which will be used repeatedly, 
		\begin{align}
		\begin{split}\label{s545}
		\sum_{i \neq \tilde k - P}^{\tilde k}& (2^n r_0)^{2i - 3} |\lam_i(t,2^n r_0)| + (2^n r_0)^{4i - d - 2} |\lam_i(t,2^n r_0)|^2
		+ r^{6i - d - 4} |\lam_i(t,2^n r_0)|^3 \\
		&+ \sum_{i \neq k - P}^{k} (2^n r_0)^{2i - 2} |\mu_i(t,2^n r_0)|^2 + (2^n r_0)^{4i - d} |\mu_i(t,2^n r_0)|^2
		+ (2^n r_0)^{6i - d -1} |\mu_i(t,2^n r_0)|^3 \\
		&\lesssim (2^n r_0)^{2(\tilde k - P - 1) - 3 + 3 \e}.
		\end{split}
		\end{align}
		Using \eqref{s542} and the relation $d = 4 \tilde k - 1$, $k = \tilde k - 1$, we estimate
		\begin{align}
		\begin{split}\label{s546}
		(2^n r_0)&^{2(\tilde k - P) - 3}|\lam_{\tilde k - P}(t, 2^n r_0)| + (2^n r_0)^{4(\tilde k - P) - d - 2}
		|\lam_{\tilde k - P}(t, 2^n r_0)|^2
		+ (2^n r_0)^{6(\tilde k - P) - d - 4}|\lam_{\tilde k - P}(t, 2^n r_0)|^3 \\
		&+ (2^n r_0)^{2(k - P) - 2}|\mu_{k - P}(t, 2^n r_0)| + (2^n r_0)^{4(k - P) - d}
		|\mu_{k - P}(t, 2^n r_0)|^2
		+ (2^n r_0)^{6(k - P) - d - 1}|\mu_{k - P}(t, 2^n r_0)|^3 \\
		&\lesssim (2^n r_0)^{2(\tilde k - P) - 3 + 3 \e}.
		\end{split} 
		\end{align}
		Inserting \eqref{s545} and \eqref{s546} into our difference estimates \eqref{s528} and \eqref{s529}, we deduce for each $n \in \N$ 
		\begin{align}
		\begin{split}\label{s547}
		|\lam_{\tilde k - P}(t,2^{n+1}r_0) - \lam_{\tilde k - P}(t,2^n r_0)| \lesssim 
		(2^n r_0)^{-2 + 3 \e}, \\
		|\mu_{k - P}(t, 2^{n+1} r_0) - \mu_{k - P}(t,2^n r_0)| \lesssim (2^n r_0)^{-1 + 3 \e}.  
		\end{split}
		\end{align}
		
		From \eqref{s547}, we deduce that 
		\begin{align*}
		&\sum_{n = 0}^\infty |\lam_{\tilde k - P}(t,2^{n+1}r_0) - \lam_{\tilde k - P}(t,2^n r_0)| \lesssim \sum_{n = 0}^\infty
		2^{(-2 + 3 \e)n} \lesssim 1, \\
		&\sum_{n = 0}^\infty |\mu_{ k - P}(t,2^{n+1}r_0) - \mu_{k - P}(t,2^n r_0)| \lesssim \sum_{n = 0}^\infty 2^{(-1 + 3 \e)n}
		\lesssim 1,
		\end{align*}
		uniformly in $t$.
		In particular, for all $t \in \R$ there exist $\al_{\tilde k - P}(t), \beta_{k - P}(t) \in \R$ such that 
		\begin{align*}
		\lim_{n \rar \infty} \lam_{\tilde k - P}(t,2^n r_0) &= \al_{\tilde k - P}(t), \\
		\lim_{n \rar \infty} \mu_{k - P}(t,2^n r_0) &= \beta_{\tilde k - P}(t), 
		\end{align*}
		with the estimates 
		\begin{align}
		\begin{split}\label{s549}
		&\left |\al_{\tilde k - P}(t) - \lam_{\tilde k - P}(t, 2^n r_0)\right | \lesssim (2^n r_0)^{-2 + 3 \e}, \\
		&\left |\beta_{k - P}(t) - \mu_{k - P}(t, 2^n r_0) \right | \lesssim (2^n r_0)^{-1 + 3 \e}. 
		\end{split}
		\end{align}
		
		Since the compactness of $\vec u_e(t)$ implies $|\lam_{\tilde k - P}(t,r_0)|$ is uniformly bounded in $t$, we have
		via \eqref{s549} 
		\begin{align*}
		|\al_{\tilde k - P}(t)| &\leq |\al_{\tilde k - P}(t) - \lam_{\tilde k - P}(t,r_0) | + |\lam_{\tilde k - P}(t,r_0)|
		\lesssim 1
		\end{align*}
		uniformly in $t$.  Thus,
		\begin{align*}
		|\lam_{\tilde k - P}(t,2^n r_0) | \lesssim 1, 
		\end{align*}
		uniformly in $t$ and $n$.  Similarly, $\beta_{k - P}(t)$ and $|\mu_{k - P}(t,2^n r_0)|$ are bounded uniformly in $t$ and $n$. In conclusion, 
		we have 
		\begin{align}\label{s548}
		|\lam_{\tilde k - P}(t,2^n r_0)| + |\mu_{k - P}(t,2^n r_0)| \lesssim 1
		\end{align}
		uniformly in $t$ and $n$. 
		
		Let $r \geq r_0$ with $2^n r_0 \leq r \leq 2^{n+1} r_0$.  If we insert \eqref{s548} back into the difference estimates
		\eqref{s528} and \eqref{s529},  
		we deduce that 
		\begin{align}
		\begin{split}\label{s551}
		|\lam_{\tilde k - P}(t,r) - \lam_{\tilde k - P}(t,2^n r_0)| \lesssim (2^n r_0)^{-2} \lesssim r^{-2}, \\
		|\mu_{k - P}(t,r) - \mu_{k -P}(t,2^n r_0)| \lesssim (2^n r_0)^{-1} \lesssim r^{-1}, 
		\end{split}
		\end{align}
		which imply the following improvements of \eqref{s549}
		\begin{align}
		\begin{split}\label{s550}
		&\left |\al_{\tilde k - P}(t) - \lam_{\tilde k - P}(t, 2^n r_0)\right | \lesssim (2^n r_0)^{-2}, \\
		&\left |\beta_{k - P}(t) - \mu_{k - P}(t, 2^n r_0) \right | \lesssim (2^n r_0)^{-1}.
		\end{split}
		\end{align}
		Finally, using \eqref{s551} and \eqref{s550} we conclude that 
		\begin{align*}
		|\al_{\tilde k - P}(t) - \lam_{\tilde k - P}(t,r) |
		\lesssim |\al_{\tilde k - P}(t) - \lam_{\tilde k - P}(t,2^n r_0) | + |\lam_{\tilde k - P}(t,r) - \lam_{\tilde k - P}(t,2^n r_0)| 
		\lesssim (2^n r_0)^{-2} \lesssim r^{-2},
		\end{align*}
		and 
		\begin{align*}
		|\beta_{k - P}(t) - \mu_{k - P}(t,r)| \lesssim 
		|\beta_{k - P}(t) - \mu_{k - P}(t,2^n r_0)| + 
		|\mu_{k - P}(t,r) - \mu_{k -P}(t,2^n r_0)| \lesssim (2^n r_0)^{-1} \lesssim r^{-1}. 
		\end{align*}
		This concludes the proof. 
	\end{proof}
	
	A corollary of Lemma \ref{l514} is the following preliminary asymptotics for $u_e$.  
	
	\begin{cor}\label{c516}
		We have 
		\begin{align}\label{s552}
		r^{-2(\tilde k - P) + d} u_e (t,r) = \al_{\tilde k - P}(t) + O(r^{-2+3\e}). 
		\end{align}
		The $O(\cdot)$ term is uniform in time.  
	\end{cor}
	
	\begin{proof}
		By Lemma \ref{l54}, \eqref{s543}, and our induction hypotheses \eqref{s542}, we have
		\begin{align*}
		r^{-2(\tilde k - P) - d} u_e(t,r) &= 
		\sum_{j = 1}^{\tilde k } \lam_j(t,r) r^{2j - 2(\tilde k - P)} \\
		&= \al_{\tilde k - P}(t) + \sum_{j \neq \tilde k - P}^{\tilde k} \lam_j(t,r)  r^{2j - 2(\tilde k - P)} + O(r^{-2}) \\
		&= \al_{\tilde k - P}(t) + O(r^{-2+3\e})
		\end{align*}
		uniformly in time.  
	\end{proof}
	
	A corollary of the proof of Lemma \ref{l514} is the following. 
	
	\begin{cor}\label{c517}
		Suppose that for all $r, r'$ with $R_1 \leq r \leq r' \leq 2r$, we have 
		\begin{align*}
		|\lam_j(t,r') - \lam_j(t,r)| \lesssim r^{-a}, 
		\end{align*}
		with $a < 0$.  Then $\lam_j(t,r)$ has a limit, $\al_j(t)$, as $r \rar \infty$.  Moreover, $\al_j(t)$ is bounded in time 
		and 
		\begin{align*}
		|\lam_j(t,r) - \al_j(t) | \lesssim r^{-a}
		\end{align*}
		uniformly in time. A similar statement holds for the $\mu_j$'s.  
	\end{cor}
	
	We will now show that 
	\begin{align*}
	\alpha_{\tilde k - P}(t) \equiv 0, \quad \beta_{k - P}(t) \equiv 0. 
	\end{align*}
	We first show that $\alpha_{\tilde k - P}(t)$ is constant in time.  
	
	\begin{lem}\label{l517}
		The function $\al_{\tk - P}(t)$ is constant in time.  From now on, we will write $\al_{\tk - P}$ in place of $\al_{\tk - P}(t)$. 
	\end{lem}
	
	\begin{proof}
		Let $t_2 \neq t_1$.  By \eqref{s532a} with $j = \tk - P$ and $j' = \tk - P - 1$, \eqref{s543}, and our induction hypotheses \eqref{s542} 
		we have 
		\begin{align*}
		|\al_{\tk - P}(t_2) - \al_{\tk - P}(t_1) | &\lesssim 
		|\lam_{\tk - P}(t_2,r) - \lam_{\tk - P}(t_1,r)| + O(r^{-2}) \\
		&\lesssim r^{-2} |\lam_{\tk - P -1}(t_2,r) - \lam_{\tk - P-1}(t_1,r)| 
		+ \sum_{m =1}^k \int_{t_1}^{t_2} r^{2m - 2(\tk - P)} |\mu_m (t,r)| dt
		+ O(r^{-2}) \\
		&\lesssim r^{-2+3\e}(1 + |t_2 - t_1|).  
		\end{align*}
		We let $r \rar \infty$ and deduce that $\al_{\tk - P}(t_2) = \al_{\tk - P}(t_1)$ as desired. 
	\end{proof}
	
	We now show that $\al_{\tk - P} = 0$.  As a consequence, we will also obtain the fact that $\beta_{k - P}(t)$ is constant in time. 
	
	\begin{lem}\label{l518}
		We have $\al_{\tk - P} = 0$ and $\beta_{k - P}(t)$ is constant in time. 
		From now on, we will write $\beta_{k - P}$ in place of $\beta_{k - P}(t)$. 
	\end{lem}
	
	\begin{proof}
		The key tool for proving both assertions is Lemma \ref{l511a}.  By \eqref{s532b} and \eqref{s544} we have 
		\begin{align*}
		\beta_{k - P}(t_2) - \beta_{k - P}(t_1) &= 
		\frac{1}{R} \int_R^{2R} \beta_{k - P}(t_2) - \beta_{k - P}(t_1) dr \\
		&= \frac{1}{R} \int_R^{2R}[\mu_{k - P}(t_2,r) - \mu_{k - P}(t_1,r) ] dr + O(R^{-1}) \\
		&= \sum_{i = 1}^k \frac{c_i c_{k - P}}{d - 2i - 2(k - P)} 
		\int_{t_1}^{t_2} I(i, k - P) + II(i, k - P) dt
		\end{align*}
		where $I(i,k-P)$ and $II(i,k-P)$ are defined as in \eqref{s532c}.  The estimates for the potential $V_e$ and nonlinearity
		$N_e$, \eqref{s57}--\eqref{s59}, along with \eqref{s552} imply 
		\begin{align*}
		\Bigl | -V_e(r) u_e + N_e(r, u_e) \Bigr | &\lesssim r^{-2 \tk - 2P - 3} + r^{-4\tk - 4P - 1} + r^{-2\tk - 6P - 3} \\
		&\lesssim r^{-2\tk - 2P - 3}. 
		\end{align*}
		Hence, using that $d = 4 \tk - 1$ and $k = \tk - 1$, we have 
		\begin{align}
		|II(i, k - P)| = \Bigl | \frac{1}{R} \int_R^{2R} r^{d - 2i - 2k - 2P} \int_r^\infty [ -V_e(\rho) u_e(t,\rho)
		+ N_e(\rho,u_e(t,\rho) ] \rho^{2i-1} d\rho dr \Bigr | \lesssim R^{-2}. \label{s553}
		\end{align}
		We now estimate the remaining term, 
		\begin{align*}
		I(i,k-P) &= -\frac{1}{R} (u_e(t,r) r^{d-2(k-P)-1}) \big |_{r = R}^{r = 2R} + (2i - 2(k-P) - 1) \frac{1}{R} 
		\int_R^{2R} u_e(t,r) r^{d- 2(k-P) -2} dr \\
		&\:- \frac{(2\ell - 2i + 3)(2i-2)}{R} \int_R^{2R} r^{d-2i-2(k-P)} \int_r^\infty u_e(t,\rho) \rho^{2i-3}d\rho dr.
		\end{align*}
		By \eqref{s552}, we have 
		\begin{align*}
		r^{d- 2(k - P) - 2} u_e(t,r) &= \al_{\tk - P} + O(r^{-2 + 3 \e}), \\
		r^{d - 2(k - P) - 2i} \int_r^\infty u_e(t,\rho) \rho^{2i - 3} d\rho &= \frac{\al_{\tk - P}}{d - 2i - 2(k - P)} + O(r^{-2 + 3\e}), 
		\end{align*}
		so that 
		\begin{align*}
		I(i,k-P) = -\frac{2(k-P)(d - 2(k-P) - 2)}{d - 2i - 2(k - P)} \al_{\tilde k - P} + O(R^{-2 + 3 \e}). 
		\end{align*}
		Thus, 
		\begin{align}\label{s554}
		\sum_{i = 1}^k \frac{c_i c_{k - P}}{d - 2i - 2(k - P)} 
		\int_{t_1}^{t_2} I(i, k - P) dt = C_0(t_2 - t_1) \al_{\tk - P}  + O(R^{-2 + 3\e}(t_2 - t_1)) 
		\end{align}
		where 
		\begin{align*}
		C_0 := - \sum_{i = 1}^k \frac{2c_i c_{k - P} (k - P) (d - 2(k-P)- 2)}{(d - 2i - 2(k - P))^2} 
		\end{align*}
		It can be shown using contour integration that $C_0 \neq 0$ (see Remark 5.29 in \cite{klls2} for the explicit value for $C_0$).  We let $R \rar \infty$
		in \eqref{s554} and deduce that 
		\begin{align}\label{s555}
		C_0(t_2 - t_1) \al_{\tk - P}  = \beta_{k - P}(t_2) - \beta_{k - P}(t_1).  
		\end{align}
		Since $|\beta_{k - P}(t)| \lesssim 1$ by Lemma \ref{l514} and $C_0 \neq 0$, we obtain 
		\begin{align*}
		\al_{\tilde k - P} = \frac{1}{C_0} \lim_{t_2 \rar \infty} \frac{\beta_{k - P}(t_2) - \beta_{k - P}(t_1)}{t_2 - t_1} = 0. 
		\end{align*}
		Thus, $\al_{\tk - P} = 0$ which by \eqref{s555} implies that $\beta_{k - P}(t)$ is constant in time. 
		
	\end{proof}
	
	We now conclude that $\beta_{k - P} = 0$.  
	
	\begin{lem}\label{l519}
		We have $\beta_{k - P} = 0$. 
	\end{lem}
	
	\begin{proof}
		By Lemma \ref{l514}, $\beta_{k -P} = \mu_{k - P}(t,R) + O(R^{-1})$ uniformly in time so that 
		\begin{align*}
		\beta_{k - P} = \frac{1}{T} \int_0^T \mu_{k - P}(t,R) dt + O(R^{-1}).
		\end{align*}
		Since $\al_{\tk - P} = 0$, we have by \eqref{s552}
		\begin{align*}
		u_e(t,r) = O(r^{-d + 2(\tilde k - P) - 2 + 3\e}), 
		\end{align*}
		uniformly in time. 
		Thus, by Lemma \ref{l54} and the relations $d = 4 \tilde k - 1$, $\tilde k = k + 1$, we have 
		\begin{align*}
		\Bigl | \int_0^T \mu_{k - P}(t,R) dt \Bigr | &\lesssim 
		\sum_{i = 1}^k R^{d - 2i - 2(k - P)} \Bigl | \int_R^\infty \int_0^T \p_t u_e(t,\rho) dt \rho^{2i - 1} d\rho \Bigr | \\
		&\lesssim 
		\sum_{i = 1}^k R^{d - 2i - 2(k - P)} \int_R^\infty  |u_e(T,\rho) - u_e(0,\rho)| \rho^{2i - 1} d\rho  \\
		&\lesssim R^{3\e}.  
		\end{align*}
		It follows that 
		\begin{align*}
		\beta_{k - P} = O ( R^{3\e} / T) + O(R^{-1}). 
		\end{align*}
		We set $R = T$ and let $T \rar \infty$ to conclude that $\beta_{k - P} = 0$ as desired. 
	\end{proof}
	
	In summary, we have now shown that if \eqref{s542} holds, then 
	\begin{align}
	\begin{split}\label{s557}
	\lam_{\tk - P}(t,r) &= O(r^{-2}), \\
	\mu_{k - P}(t,r) &= O(r^{-1}),
	\end{split}
	\end{align}
	uniformly in time.  We will now insert \eqref{s557} back into the difference estimates \eqref{s528} and \eqref{s529}
	to obtain \eqref{s542} for $P + 1$. 
	
	\begin{lem}\label{l520}
		Assume \eqref{s542} is true for $0 \leq P \leq k - 1$.  Then \eqref{s542} holds for $P + 1$.  
	\end{lem}
	
	\begin{proof}
		We recall that by \eqref{s545}, we have for all $r > R_1$
		\begin{align}
		\begin{split}\label{s558}
		\sum_{i \neq \tilde k - P}^{\tilde k}& r^{2i - 3} |\lam_i(t,r)| + r^{4i - d - 2} |\lam_i(t,r)|^2
		+ r^{6i - d - 4} |\lam_i(t,r)|^3 \\
		&+ \sum_{i \neq k - P}^{k} r^{2i - 2} |\mu_i(t,r)|^2 + r^{4i - d} |\mu_i(t,r)|^2
		+ r^{6i - d -1} |\mu_i(t,r)|^3, \\
		&\lesssim r^{2(\tilde k - P - 1) - 3 + 3 \e}
		\end{split}
		\end{align}
		with the main contribution coming from the linear terms. By \eqref{s557}, we have for all $r > R_1$
		\begin{align}
		\begin{split}\label{s558b}
		r&^{2(\tilde k - P) - 3}|\lam_{\tilde k - P}(t, r)| + r^{4(\tilde k - P) - d - 2}
		|\lam_{\tilde k - P}(t, r)|^2
		+ r^{6(\tilde k - P) - d - 4}|\lam_{\tilde k - P}(t, r)|^3 \\
		&+ r^{2(k - P) - 2}|\mu_{k - P}(t, r)| + r^{4(k - P) - d}
		|\mu_{k - P}(t, r)|^2
		+ r^{6(k - P) - d - 1}|\mu_{k - P}(t, r)|^3 \\
		&\lesssim r^{2(\tilde k - P-1) - 3 + 3 \e}
		\end{split} 
		\end{align}
		with the main contribution coming from the linear terms. Thus, inserting \eqref{s558} and \eqref{s558b} into our difference estimate \eqref{s528}, we have for all $R_1 \leq r \leq r' \leq 2r$, 
		\begin{align}\label{s559}
		|\lam_j(t,r') - \lam_j(t,r) | \lesssim r^{2(\tk - (P + 1) - j) - 2 + 3\e}. 
		\end{align}
		By our induction hypotheses \eqref{s542}, if $\tilde k - P < j \leq k - 1$, we have $\lam_j(t,r) \rar 0$.  By 
		Corollary \ref{c517} we then deduce that 
		\begin{align*}
		|\lam_j(t,r) | \lesssim r^{2(\tilde k - (P +1) - j)  -2 + 3 \e}
		\end{align*}
		uniformly in time.  If $j = \tilde k - (P +1)$, by \eqref{s559} and Corollary \ref{c517} we also deduce that 
		$$|\lam_{\tilde k - (P + 1)}(t,r)| \lesssim 1 \lesssim r^{\e}$$ uniformly in time.  Finally, if $j > \tk - (P + 1)$, 
		we have $2(\tk - P - 1 - j) - 2 + 3\e > \e$ so that by \eqref{s559} and Corollary \ref{cor511} 
		\begin{align*}
		|\lam_j(t,r)| \lesssim r^{2(\tk - (P + 1) - j) - 2 + 3\e}.
		\end{align*}
		In conclusion, we have shown that 
		\begin{align*}
		|\lam_j(t,r)| &\lesssim r^{2(\tilde k - (P+1) -j) - 2 + 3 \e}, \quad \forall 1 \leq j \leq \tilde k \mbox{ with } 
		j \neq \tilde k - (P+1), \\
		|\lam_{\tilde k - (P+1)}(t,r)| &\lesssim r^{\e},
		\end{align*}
		uniformly in time. A similar argument establishes 
		\begin{align*}
		|\mu_j(t,r)| &\lesssim r^{2(k - (P+1) -j) - 1 + 3 \e}, \quad \forall 1 \leq j \leq k \mbox{ with } j \neq k - (P+1), \\
		|\mu_{k - (P+1)}(t,r)| &\lesssim r^{\e}.
		\end{align*}
		This proves Lemma \ref{l520}. 
	\end{proof} 
	
	By Lemma \ref{l520} and induction, we have proved Proposition \ref{p513}.
\end{proof}

The final step in proving Proposition \ref{p58} is to establish that $\lam_1(0,r)$ has a limit as $r \rar \infty$.  In what follows 
we denote $\lam_j(r) = \lam_j(0,r)$ and $\mu_j(r) = \mu_j(0,r)$.  

\begin{lem}\label{l521}
	There exists $\al \in \R$ such that 
	\begin{align}\label{s562}
	|\lam_1(r) - \al| = O(r^{-2}).  
	\end{align}
	Moreover, we have the slightly improved decay rates
	\begin{align}
	\begin{split}\label{s563}
	|\lam_j(r)| &\lesssim r^{-2j}, \quad 1 < j \leq \tilde k, \\
	|\mu_j(r)| &\lesssim r^{-2j - 1}, \quad 1 \leq j \leq k.
	\end{split} 
	\end{align}
\end{lem}

\begin{proof}
	By \eqref{s541},
	\begin{align} 
	\begin{split}\label{s564}
	\sum_{i = 2}^{\tilde k}& r^{2i - 3} |\lam_i(t,r)| + r^{4i - d - 2} |\lam_i(t,r)|^2
	+ r^{6i - d - 4} |\lam_i(t,r)|^3 \\
	&+ \sum_{i = 1}^{k} r^{2i - 2} |\mu_i(t,r)|^2 + r^{4i - d} |\mu_i(t,r)|^2
	+ r^{6i - d -1} |\mu_i(t,r)|^3, \\
	&\lesssim r^{-3 + 3\e}
	\end{split}
	\end{align}
	with the main contribution coming from the linear terms, and 
	\begin{align}
	\begin{split}\label{s565}
	r&^{2 - 3}|\lam_{1}(t, r)| + r^{4 - d - 2}
	|\lam_{1}(t, r)|^2
	+ r^{6 - d - 4}|\lam_{1}(t, r)|^3 
	\lesssim r^{-1 + \e}.
	\end{split} 
	\end{align}
	We insert \eqref{s564} and \eqref{s565} into our difference estimate \eqref{s528} and conclude that for all $r,r'$ 
	with $R_1 < r \leq r' \leq 2r$
	\begin{align*}
	|\lam_1(r') - \lam_1(r) | \lesssim r^{-2 + \e}.  
	\end{align*}
	By Corollary \ref{c517}, we deduce that there exists $\alpha \in \R$ such that 
	\begin{align*}
	|\lam_1(r) - \al | \lesssim r^{-2+\e}. 
	\end{align*}
	In particular, $|\lam_1(r)| \lesssim 1$.  This improves the estimate \eqref{s565} to 
	\begin{align}
	\begin{split}\label{s566}
	r&^{2 - 3}|\lam_{1}(t, r)| + r^{4 - d - 2}
	|\lam_{1}(t, r)|^2
	+ r^{6 - d - 4}|\lam_{1}(t, r)|^3 
	\lesssim r^{-1}.
	\end{split} 
	\end{align}
	We plug \eqref{s564} and \eqref{s566} back into our difference estimate \eqref{s528} and conclude that for all $r,r'$ 
	with $R_1 < r \leq r' \leq 2r$
	\begin{align*}
	|\lam_1(r') - \lam_1(r) | \lesssim r^{-2}.  
	\end{align*}
	Thus, 
	\begin{align*}
	|\lam_1(r) - \al| \lesssim r^{-2}. 
	\end{align*}
	By \eqref{s564} and \eqref{s566} and the difference estimate \eqref{s528} we conclude that for the other coefficients, 
	for all $r,r'$ with $R_1 < r \leq r' \leq 2r$
	\begin{align}
	|\lam_j(r') - \lam_j(r) | &\lesssim r^{-2j}, \\
	|\mu_j(r') - \mu_j(r) | &\lesssim r^{-2j-1}. 
	\end{align}
	By \eqref{s541}, these coefficients go to 0 as $r \rar \infty$ so that by Corollary \ref{c517} we conclude that 
	\begin{align*}
	|\lam_j(r) | &\lesssim r^{-2j}, \\
	|\mu_j(r) | &\lesssim r^{-2j-1}.
	\end{align*}
	This completes the proof. 
\end{proof}

\begin{proof}[Proof of Proposition \ref{p58}]
	By \eqref{s562}, \eqref{s563} and Lemma \ref{l54}
	\begin{align*}
	r^{d-2} u_e(0,r) &= \sum_{j = 1}^{\tk} \lam_j(r) r^{2j - 2} \\
	&= \lam_1(r) + \sum_{j = 2}^{\tk} \lam_j(r) r^{2j - 2} \\
	&= \al + O(r^{-2})
	\end{align*}
	as well as
	\begin{align*}
	\int_r^{\infty} \p_t u_e(0, \rho) \rho^{2i -1} d \rho &= \sum_{j = 1}^k \mu_j(r) \frac{r^{2i + 2j - d}}{d - 2i - 2j} \\
	&= O(r^{2i - d - 1})
	\end{align*}
	as desired.
\end{proof}

We now establish Proposition \ref{p58} in the case that $d = 5, 9, 13,\ldots,$ i.e. when $d = 2\ell + 3$ with 
$\ell$ \textbf{odd}.  The case $\ell = 1$, $d = 5$, is contained in \cite{cpr}.  When $\ell$ is odd, we have the identities
\begin{align*}
k = \tilde k = \frac{\ell + 1}{2}, \quad d = 4k + 1. 
\end{align*}
The proof of Proposition \ref{p58} for when $\ell$ is odd is very similar to the case when $\ell$ is even but contains
subtleties because of the above identities.  In particular, there is an extra $\mu$ coefficient, 
$\mu_k$, which must be dealt with before we can proceed to showing $\lam_j, \mu_{j-1}$ tend to 0 by induction. 

We first establish an $\e$--growth estimate for the coefficients. 
\begin{lem}\label{l522}
	Let $\epsilon > 0$ be fixed and sufficiently small.  Then as long as $\delta_1$ as in Lemma 
	\ref{l59} is sufficiently small, we have uniformly in $t$, 
	\begin{align}
	\begin{split}\label{s567a}
	|\lam_{k}(t,r)| &\lesssim r^{\e}, \\
	|\mu_k(t,r)| &\lesssim r^{\e}, \\
	|\lam_i(t,r)| &\lesssim r^{2 k - 2j - 1 + 3\e}, \quad \forall 1 \leq i < k, \\
	|\mu_i(t,r)| &\lesssim r^{2 k - 2j - 2 + 3\e}, \quad \forall 1 \leq i < k.
	\end{split}
	\end{align}
\end{lem}

\begin{proof}
	Let $r > R_1$. By Corollary 
	\ref{c510} we have, 
	\begin{align}
	\begin{split}\label{s567}
	|\lam_j(t,2r)| &\leq  (1 + C \de_1) |\lam_j(t,r)| + 
	C \de_1 \left ( \sum_{i = 1}^{k} r^{2i-2j} |\lam_i(t,r)| +
	\sum_{i = 1}^k r^{2i-2j+1} |\mu_i(t,r)| \right ), \\
	|\mu_j(t,2r)| &\leq  (1 + C \de_1) |\mu_j(t,r)| + 
	r^{-1} C \de_1 \left ( \sum_{i = 1}^{k} r^{2i-2j} |\lam_i(t,r)| +
	\sum_{i = 1}^k r^{2i-2j+1} |\mu_i(t,r)| \right ).
	\end{split} 
	\end{align}
	Fix $r_0 > R_1$ and define 
	\begin{align*}
	b_n :=   \sum_{i = 1}^{k} (2^n r_0)^{2i-2k-1} |\lam_i(t,2^n r_0)| +
	\sum_{i = 1}^k (2^n r_0)^{2i-2k} |\mu_i(t,2^n r_0)|. 
	\end{align*}
	Then by \eqref{s567}
	\begin{align*}
	b_{n+1} \leq ( 1 + 2Ck \de_1 ) b_n.
	\end{align*} 
	By iterating we obtain 
	\begin{align*}
	b_n \leq ( 1 + 2Ck \de_1 )^n b_0 
	\end{align*}
	Choose $\de_1$ so small so that $1 + 2Ck \de_1 < 2^{\e}$.  By the compactness of $\vec u_e(t)$, 
	$b_0 \lesssim 1$ uniformly in $t$, and we conclude that 
	\begin{align*}
	b_n \leq 2^{n \e} b_0 \lesssim 2^{n \e}.
	\end{align*}
	By our definition of $b_n$ it follows that 
	\begin{align}\label{s568}
	|\lam_i(t,2^n r_0)| \lesssim (2^n r_0)^{2k - 2i + 1 + \e}, \quad |\mu_i(t,2^n r_0)| \lesssim (2^n r_0)^{2k - 2i + \e},
	\end{align}
	which is an improvement of \eqref{s527}. 
	
	As in the proof of Lemma \ref{l511}, we insert \eqref{s568} back into our difference estimates \eqref{s528} and \eqref{s529}
	and conclude that  
	\begin{align}\label{s540}
	|\lam_j(t,2^{n+1} r_0) - \lam_j(t,2^n r_0)| \leq C \de_1 |\lam_j(2^n r_0)| + C (2^n r_0)^{2 k - 2j - 1 +  3\e}, \\
	|\mu_j(t,2^{n+1} r_0)- \mu_j(t,2^n r_0)| \leq C \de_1 |\mu_j(2^n r_0)| + C (2^n r_0)^{2 k - 2j - 2 + 3\e},
	\end{align}
	with the dominant contribution coming from the cubic terms. By Corollary \ref{cor511} we conclude that uniformly in $t$
	\begin{align*}
	|\lam_{k}(t,r)| &\lesssim r^{\e}, \\
	|\mu_k(t,r)| &\lesssim r^{\e}, \\
	|\lam_i(t,r)| &\lesssim r^{2\tilde k - 2j - 1 + 3\e}, \quad \forall 1 \leq i < k, \\
	|\mu_i(t,r)| &\lesssim r^{2\tilde k - 2j - 2 + 3\e}, \quad \forall 1 \leq i < k, 
	\end{align*} 
	as desired. 
\end{proof}

We now turn to showing that the extra term $\mu_k$ goes to 0 as $r \rar \infty$. 

\begin{lem}\label{l523}
	There exists a bounded function $\beta_k(t)$ on $\R$ such that 
	\begin{align}\label{s569}
	|\mu_k(t,r) - \beta_k(t) | \lesssim r^{-2},
	\end{align}
	uniformly in $t$. 
\end{lem}

\begin{proof}
	We insert \eqref{s566} into the difference estimate \eqref{s529} with $j=k$ and obtain for all $R_1 < r_1 < r' < 2r$
	\begin{align*}
	|\mu_k(t,r') - \mu_k(t,r)| \lesssim r^{-2+3\e}.
	\end{align*}
	The dominant contribution in the difference estimate \eqref{s529} comes from the cubic term $|\mu_k|^3$. By Corollary 
	\ref{c517}, we conclude that there exists a bounded function $\beta_k(t)$ such that 
	\begin{align*}
	|\mu_k(t,r) - \beta_k(t)| \lesssim r^{-2+3\e}. 
	\end{align*}
	uniformly in $t$. 
	In particular, $|\mu_k(t,r)|\lesssim 1$ uniformly in $t$ and $r$.  Using this information, we can improve the difference
	estimate \eqref{s529} with $j = k$ to 
	\begin{align*}
	|\mu_k(t,r') - \mu_k(t,r)| \lesssim r^{-2}
	\end{align*}
	and conclude that 
	\begin{align*}
	|\mu_k(t,r) - \beta_k(t)| \lesssim r^{-2}
	\end{align*}
	uniformly in $t$ as desired. 
\end{proof}

\begin{lem}\label{l524}
	We have $\beta_k(t) \equiv 0$. 
\end{lem}

\begin{proof}
	The proof is in similar spirit to the proofs of Lemmas \ref{l518} and \ref{l519}.  We first note that by Lemma 
	\ref{l54}, \eqref{s567a}, and the relation $d = 4k + 1$, we have 
	\begin{align}\label{s570}
	|u_e(t,r)| = \Bigl | \sum_{j = 1}^k \lam_j(t,r) r^{2j-d} \Bigr | \lesssim r^{2k - d+\e} = r^{-2k -1 + \e}. 
	\end{align}
	By \eqref{s569} and \eqref{s532b}
	\begin{align*}
	\beta_{k}(t_2) - \beta_{k }(t_1)
	&= \frac{1}{R} \int_R^{2R}[\mu_{k}(t_2,r) - \mu_{k}(t_1,r) ] dr + O(R^{-2}) \\
	&= \sum_{i = 1}^k \frac{c_i c_{k}}{d - 2i - 2k} 
	\int_{t_1}^{t_2} I(i, k) + II(i, k) dt
	\end{align*}
	where $I(i,k)$ and $II(i,k)$ are defined as in \eqref{s532c}.  The estimates for the potential $V_e$ and nonlinearity
	$N_e$, \eqref{s57}--\eqref{s59}, along with \eqref{s570} imply 
	\begin{align*}
	\Bigl | -V_e(r) u_e + N_e(r, u_e) \Bigr |
	\lesssim r^{-2k - 5 + \e }. 
	\end{align*}
	Using that $d = 4 k + 1$ we have 
	\begin{align*}
	|II(i, k - P)| = \Bigl | \frac{1}{R} \int_R^{2R} r^{d - 2i - 2k - 2P} \int_r^\infty [ -V_e(\rho) u_e(t,\rho)
	+ N_e(\rho,u_e(t,\rho) ] \rho^{2i-1} d\rho dr \Bigr | \lesssim R^{-4+\e}.
	\end{align*}
	We now estimate the remaining term, 
	\begin{align*}
	I(i,k) &= -\frac{1}{R} (u_e(t,r) r^{d-2k-1}) \big |_{r = R}^{r = 2R} + (2i - 2k - 1) \frac{1}{R} 
	\int_R^{2R} u_e(t,r) r^{d- 2k -2} dr \\
	&\:- \frac{(2\ell - 2i + 3)(2i-2)}{R} \int_R^{2R} r^{d-2i-2k} \int_r^\infty u_e(t,\rho) \rho^{2i-3}d\rho dr.
	\end{align*}
	Using \eqref{s570}, it is simple to conclude that 
	\begin{align*}
	|I(i,k)| \lesssim R^{-2 + \e}. 
	\end{align*}
	Thus, 
	\begin{align*}
	\beta_k(t_2) - \beta_k(t_1) = O(R^{-2+ \e})(1 + |t_2 - t_1|).
	\end{align*}
	We let $R \rar \infty$ and conclude that $\beta_k(t_2) = \beta_k(t_1)$ as desired.
	
	We now write $\beta_k$ in place of $\beta_k(t)$.  By the previous paragraph, we have that 
	\begin{align*}
	\beta_k = \mu_k(t,R) + O(R^{-2}) 
	\end{align*}
	where the $O(\cdot)$ term is uniform in time. Integrating the previous expression from 0 to $T$ and dividing by $T$ yields
	\begin{align*}
	\beta_{k} = \frac{1}{T} \int_0^T \mu_{k}(t,R) dt + O(R^{-2}).
	\end{align*}
	By Lemma \ref{l54}, \eqref{s570}, and the relations $d = 4 k + 1$, we have 
	\begin{align*}
	\Bigl | \int_0^T \mu_{k}(t,R) dt \Bigr | &\lesssim 
	\sum_{i = 1}^k R^{d - 2i - 2k} \Bigl | \int_R^\infty \int_0^T \p_t u_e(t,\rho) dt \rho^{2i - 1} d\rho \Bigr | \\
	&\lesssim 
	\sum_{i = 1}^k R^{d - 2i - 2k} \int_R^\infty  |u_e(T,\rho) - u_e(0,\rho)| \rho^{2i - 1} d\rho  \\
	&\lesssim R^{\e}.  
	\end{align*}
	It follows that 
	\begin{align*}
	\beta_{k} = O ( R^{3\e} / T) + O(R^{-2}). 
	\end{align*}
	We set $R = T$ and let $T \rar \infty$ to conclude that $\beta_{k} = 0$ as desired. 
\end{proof}

\begin{lem}\label{l525}
	Let $\epsilon > 0$ be fixed and sufficiently small.  Then as long as $\delta_1$ as in Lemma 
	\ref{l59} is sufficiently small, we have uniformly in $t$, 
	\begin{align}
	\begin{split}\label{s571}
	|\lam_{k}(t,r)| &\lesssim r^{\e}, \\
	|\mu_{k-1}(t,r)| &\lesssim r^{\e}, \\
	|\lam_i(t,r)| &\lesssim r^{2k - 2j - 2 + \e}, \quad \forall 1 \leq i < k, \\
	|\mu_i(t,r)| &\lesssim r^{2k - 2j - 3 + \e}, \quad \forall 1 \leq i \leq < k, i \neq k -1. 
	\end{split}
	\end{align}
\end{lem}

\begin{proof}
	We first establish 
	\begin{align*}
	|\mu_k(t,r)| \lesssim r^{-3 + \e},
	\end{align*}
	uniformly in time.  
	By \eqref{s567a}, we have uniformly in time 
	\begin{align}
	\begin{split}\label{s572}
	\sum_{i = 1}^k & r^{2i - 3} |\lam_i(t,r)| + r^{4i - d - 2}|\lam_i(t,r)|^2 + r^{6i - d -4} |\lam_i(t,r)|^3 \\
	&+ \sum_{i = 1}^{k-1} r^{2i - 2} |\mu_i(t,r)| + r^{4i - d} |\mu_i(t,r)|^2 + r^{6i - d -1} |\mu_i(t,r)|^3 
	\lesssim r^{2k - 3 + \e},
	\end{split}
	\end{align}
	where the dominant contribution comes from the linear term involving $\lam_k$.  By \eqref{s569} and the fact that $\beta_k = 0$, 
	we have 
	\begin{align}\label{s573}
	r^{2k - 2} |\mu_k(t,r)| + r^{4k - d} |\mu_k(t,r)|^2 + r^{6k - d -1} |\mu_k(t,r)|^3 \lesssim r^{2k - 4}, 
	\end{align}
	uniformly in time.  Inserting \eqref{s572} and \eqref{s573} into the difference estimate \eqref{s529} with $j = k$ implies that for all 
	$r,r'$ with $R_1 \leq r \leq r' \leq 2r$, we have
	\begin{align*}
	|\mu_k(t,r') - \mu_k(t,r)| \lesssim r^{-3 + \e} 
	\end{align*}
	uniformly in time.   Since $\lim_{r \rar \infty} \mu_k(t,r) = 0$, we conclude by Corollary \ref{c517} that 
	\begin{align}\label{s574}
	|\mu_k(t,r)| \lesssim r^{-3 + \e}
	\end{align}
	uniformly in time. 
	
	We now establish the other estimates in \eqref{s571}.  Fix $r_0 > R_1$.  By \eqref{s574}
	\begin{align*}
	(2^n r_0)^{2k - 2} |\mu_k(t,2^n r_0)| + (2^n r_0)^{4k - d} |\mu_k(t,2^n r_0)|^2 + (2^n r_0)^{6k - d -1} 
	|\mu_k(t,2^n r_0)|^3 
	\lesssim (2^n r_0)^{2k - 5 + \e}
	\end{align*}
	uniformly in time.  This estimate along with \eqref{s572} and the difference estimate \eqref{s528} imply for all $1 \leq j \leq k$
	\begin{align*}
	|\lam_j(t,2^{n+1} r_0)| \leq (1 + C \de_1) |\lam_j(t,2^n r_0)| + C (2^n r_0)^{2k - 2j - 2 + \e},
	\end{align*}
	uniformly in time.  By Corollary \ref{cor511}, we conclude that 
	\begin{align*}
	|\lam_j(t,r)| &\lesssim r^{2k - 2j - 2 + \e}, \quad \forall 1 \leq j < k, \\
	|\lam_k(t,r)| &\lesssim r^{\e}, 
	\end{align*}
	uniformly in time.  A similar argument establishes the remaining bounds in \eqref{s571} involving the $\mu_j$'s.  
	
\end{proof}

As in the case that $\ell$ is even, we use Lemma \ref{l525} as the base case for an induction argument.  In particular, we 
will prove the following.

\begin{ppn}\label{p526}
	Suppose $d = 5,9,13,\ldots$ and $\e$, $\de_1$,$r_0$ are as in Lemma \ref{l525}.  For $P = 0, 1, \ldots, k-1$ the following estimates hold uniformly in time:
	\begin{align}
	\begin{split}\label{s575}
	|\lam_j(t,r)| &\lesssim r^{2(k - P -j) - 2 +  \e}, \quad \forall 1 \leq j \leq k \mbox{ with } j \neq  k - P, \\
	|\lam_{k - P}(t,r)| &\lesssim r^{\e},\\
	|\mu_j(t,r)| &\lesssim r^{2(k - P -j) - 3 +  \e}, \quad \forall 1 \leq j \leq k \mbox{ with } j \neq k - P-1, \\
	|\mu_{k - P-1}(t,r)| &\lesssim r^{\e}.
	\end{split}
	\end{align}
\end{ppn}

In we take $P = k-1$ in Proposition \ref{p526}, then we obtain the following.

\begin{ppn}\label{p527}
	With the same hypotheses as in Proposition \ref{p526}, the following estimates hold uniformly in time:
	\begin{align}
	\begin{split}\label{s576}
	|\lam_j(t,r)| &\lesssim r^{-2j + \e}, \quad \forall 1 < j \leq k, \\
	|\lam_1(t,r)| &\lesssim r^{\e}, \\
	|\mu_j(t,r)| &\lesssim r^{-2j -1 + \e}, \quad \forall 1 \leq j \leq k.
	\end{split}
	\end{align}
\end{ppn}

\begin{proof}[Proof of Proposition \ref{p526}]
	The proof of Proposition \ref{p526} is nearly identical to the proof of Proposition \ref{p513}.  Therefore, we will only
	outline the main steps of the proof and refer the reader to the proofs given for the case that $\ell$ is even for the details.
	The proof is by induction on $P$.  The case $P = 0$ is covered in Lemma \ref{l525}.  
	We now assume that \eqref{s575} holds for all $P$ with $0 \leq P < k-1$.  
	\begin{description}
		\item[Step 1]  There exist bounded functions $\al_{k-P}(t)$ and $\beta_{k-P-1}(t)$ defined on $\R$ such that 
		\begin{align*}
		|\lam_{k - P}(t,r) - \al_{k-P}(t)| &\lesssim r^{-2}, \\
		|\mu_{k - P - 1}(t,r) - \beta_{k-P-1}(t)| &\lesssim r^{-1},
		\end{align*}
		uniformly in $t$. For details, see the proof of Lemma \ref{l514}. 
		
		\item[Step 2] We have 
		\begin{align*}
		r^{d - 2(k-P)} u_e(t,r) = \al_{k-P}(t) + O ( r^{-2 + \e} ), 
		\end{align*}
		where the $O(\cdot)$ term is uniform in time.  For details, see the proof of Corollary \ref{c516}. 
		
		\item[Step 3]  The function $\al_{k-P}(t)$ is constant in time and from now on we write $\al_{k-p}$ in place of 
		$\al_{k - P}(t)$.  For details, see the proof Lemma \ref{l517}.
		
		\item[Step 4]  We have $\al_{k-P} = 0$ and $\beta_{k - P - 1}(t)$ is constant in time.  From now on we write $\beta_{k - P - 1}$ 
		in place of $\beta_{k - P - 1}(t)$.  For details, see the proof of Lemma \ref{l518}.  
		
		\item[Step 5]  We have $\beta_{k - P - 1} = 0$.  For details, see the proof of Lemma \ref{l519}.  
		
		From Steps 1--5, we conclude that 
		\begin{align}
		\begin{split}\label{s577}
		\lam_{k - P}(t) &= O(r^{-2}), \\
		\mu_{k - P - 1}(t) &= O(r^{-1}),
		\end{split}
		\end{align}
		uniformly in time.  Inserting \eqref{s575} and \eqref{s577} into the difference estimates \eqref{s528} and \eqref{s529},
		we conclude
		that the following holds. 
		
		\item[Step 6]  If \eqref{s575} holds for all $0 \leq P < k-1$, then \eqref{s575} holds for $P + 1$.  For details, 
		see the proof of Lemma \ref{l520}.  
	\end{description}
	
	By induction and Step 6, we have proved Proposition \ref{p526}.  
	
\end{proof}

As in the case that $\ell$ is even, from Proposition \ref{p527} we deduce the following behavior for $\lam_1$. 

\begin{lem}\label{l528}
	There exists $\al \in \R$ such that 
	\begin{align}\label{s578}
	|\lam_1(r) - \al| = O(r^{-2}).  
	\end{align}
	Moreover, we have the slightly improved decay rates
	\begin{align}
	\begin{split}\label{s579}
	|\lam_j(r)| &\lesssim r^{-2j}, \quad \forall 1 < j \leq k, \\
	|\mu_j(r)| &\lesssim r^{-2j - 1}, \quad \forall 1 \leq j \leq k.
	\end{split} 
	\end{align}
\end{lem}

\begin{proof}
	The proof is identical to the proof of Lemma \ref{l521}.  
\end{proof}

The proof of Proposition \ref{p58} for the case that $\ell$ odd is identical to the case that $\ell$ is even and we omit it. 

\subsubsection*{Step 3:  Conclusion of the Proof of Proposition \ref{p53}}

Let $\al$ be as in Proposition \ref{p58}.  We now show that there exists a unique static solution $U_+$ to \eqref{s52e} such that 
$\vec u(0) = (U_+, 0)$ on $r \geq \eta$ (where $U_+$ does not depend on $\eta$).  We distinguish two cases: $\al = 0$ and $\al \neq 0$.  For the case $
\al = 0$, we will show that $\vec u(0) = (0,0)$ on $r \geq \eta$.  We first show that if $\al = 0$, then 
$\vec u(0,r)$ is compactly supported.  

\begin{lem}\label{l529}
	Let $\vec u_e$ be as in Proposition \ref{p53}, and let $\al$ be as in Proposition \ref{p58}.  If $\al = 0$, then 
	$\vec u(0,r)$ is compactly supported in $r \in (\eta,\infty)$.   
\end{lem}

\begin{proof}
	If $\al = 0$, then by Lemma \ref{l521} and Lemma \ref{l528}, we have
	\begin{align}
	\begin{split}\label{s580}
	|\lam_j(r)| &\lesssim r^{-2j}, \quad \forall 1 \leq j \leq \tk, \\
	|\mu_j(r)| &\lesssim r^{-2j - 1}, \quad \forall 1 \leq j \leq k.  
	\end{split}
	\end{align}
	Thus, there exists $C_1$ such that for all $r \geq \eta$ 
	\begin{align}\label{s581}
	\sum_{j = 1}^{\tk} r^{2j} |\lam_j(r)| + \sum_{j = 1}^k r^{2j + 1} |\mu_j(r)| \leq C_1.  
	\end{align}
	Fix $r_0 > R_1$.  By Corollary \ref{c510}, we have 
	\begin{align*}
	|\lam_j(2^{n+1} r_0)| \geq ( 1 &- C \de_1 ) |\lam_j(2^n r_0)| \\&- C \de_1 (2^n r_0)^{-2j} 
	\left [ 
	\sum_{i = 1}^{\tk} (2^n r_0)^{2i} |\lam_i(2^n r_0)| + \sum_{i = 1}^k (2^n r_0)^{2i + 1} |\mu_i(2^n r_0)|,   
	\right ], 
	\end{align*}
	and 
	\begin{align*}
	|\mu_j(2^{n+1} r_0)| \geq ( 1 &- C \de_1 ) |\mu_j(2^n r_0)| \\&- C \de_1 (2^n r_0)^{-2j-1} 
	\left [ 
	\sum_{i = 1}^{\tk} (2^n r_0)^{2i} |\lam_i(2^n r_0)| + \sum_{i = 1}^k (2^n r_0)^{2i + 1} |\mu_i(2^n r_0)|,   
	\right ].  
	\end{align*}
	We conclude that 
	\begin{align*}
	\sum_{i = 1}^{\tk} (2^{n+1} r_0)^{2i} |\lam_i(2^{n+1} r_0)|&+ \sum_{i = 1}^k (2^{n+1} r_0)^{2i + 1} |\mu_i(2^{n+1} r_0)|\\
	&\geq 4 \Bigl ( 1 - C \de_1 (\tilde k + k + 1) 2^{2k + 1} \Bigr ) \left [
	\sum_{i = 1}^{\tk} (2^n r_0)^{2i} |\lam_i(2^n r_0)| + \sum_{i = 1}^k (2^n r_0)^{2i + 1} |\mu_i(2^n r_0)| \right ].
	\end{align*}
	If we fix $\de_1$ so small so that $C \de_1 (\tilde k + k + 1) 2^{2k + 1} < \frac{1}{2}$, then we conclude that 
	\begin{align*}
	\sum_{i = 1}^{\tk} (2^{n+1} r_0)^{2i} |\lam_i(2^{n+1} r_0)|&+ \sum_{i = 1}^k (2^{n+1} r_0)^{2i + 1} |\mu_i(2^{n+1} r_0)|\\
	&\geq 2 \left [
	\sum_{i = 1}^{\tk} (2^n r_0)^{2i} |\lam_i(2^n r_0)| + \sum_{i = 1}^k (2^n r_0)^{2i + 1} |\mu_i(2^n r_0)| \right ]. 
	\end{align*}
	Iterating, we conclude that for all $n \geq 0$, 
	\begin{align*}
	\sum_{i = 1}^{\tk} (2^n r_0)^{2i} |\lam_i(2^n r_0)| + \sum_{i = 1}^k (2^n r_0)^{2i + 1} |\mu_i(2^n r_0)|
	\geq 2^n \left [
	\sum_{i = 1}^{\tk} (r_0)^{2i} |\lam_i(r_0)| + \sum_{i = 1}^k (r_0)^{2i + 1} |\mu_i(r_0)| \right ].
	\end{align*}
	By \eqref{s581}, we obtain for all $n \geq 0$, 
	\begin{align*}
	\sum_{j = 1}^{\tk} (r_0)^{2i} |\lam_i(r_0)| + \sum_{i = 1}^k (r_0)^{2i + 1} |\mu_i(r_0)| \leq 2^{-n} C_1. 
	\end{align*}
	Letting $n \rar \infty$ implies that 
	\begin{align*}
	\lam_j(r_0) = \mu_i(r_0) = 0, \quad \forall 1 \leq j \leq \tk, 1 \leq i \leq k.  
	\end{align*}
	By Lemma \ref{l54} and Lemma \ref{l55}, it follows that $\| \vec u_e(0) \|_{\h(r \geq r_0)} = 0$.  Thus, 
	$(\p_r u_{e,0}, u_{e,1})$ is compactly supported in $(\eta,\infty)$.  Since
	\begin{align*}
	\lim_{r \rar \infty} u_{e,0}(r) = 0, 
	\end{align*}
	we conclude that $(u_{e,0},u_{e,1})$ is compactly supported as well.  
\end{proof}

\begin{lem}\label{l530}
	Let $\vec u_e$ be as in Proposition \ref{p53}, and let $\al$ be as in Proposition \ref{p58}.  If $\al = 0$, then 
	$\vec u(0,r) = (0,0)$ on $r \geq \eta$.   
\end{lem}

\begin{proof}
	If $\al = 0$, then by Lemma \ref{l529}, $\vec u_e(0) = 
	(u_{e,0}(r), u_{e,1}(r))$ is compactly supported in $(\eta,\infty)$.  Thus, we may define 
	\begin{align*}
	\rho_0 := \inf \left \{ \rho : \| \vec u_e(0) \|_{\h(r \geq \rho)} = 0 \right \} < \infty.  
	\end{align*}
	We now argue by contradiction and assume that $\rho_0 > \eta$.  Let $\e > 0$ to be fixed later, and choose $\rho_1 \in (\eta,\rho_0)$ 
	close to $\rho_0$ so that 
	\begin{align}\label{s582}
	0 < \| \vec u_e(0) \|_{\h(r \geq \rho_1)} < \e, 
	\end{align}
	where $\de_1$ is as in Lemma \ref{l59}.
	
	By Lemma \ref{l54}, we have 
	\begin{align*}
	0 &= \| \vec u_e(0) \|_{\h(r \geq \rho_0)}^2 \\
	&\simeq 
	\sum_{i = 1}^{\tilde k} \Bigl ( \lam_i(\rho_0) \rho_0^{2i - \frac{d+2}{2}} \Bigr )^2 
	+ \sum_{j = 1}^{k} \Bigl ( \mu_j(\rho_0) \rho_0^{2i - \frac{d}{2}} \Bigr )^2,  \\
	&\:+ \int_{\rho_0}^\infty
	\sum_{i = 1}^{\tilde k} \Bigl ( \p_r \lam_i(r) r^{2i - \frac{d+1}{2}} \Bigr )^2 
	+ \sum_{j = 1}^{k} \Bigl ( \p_r \mu_j(r) r^{2i - \frac{d-1}{2}} \Bigr )^2 dr. 
	\end{align*}
	Thus, $\lam_j(\rho_0) = \mu_i(\rho_0) = 0$ for all $1 \leq j \leq \tk, 1 \leq i \leq k$.  
	
	A simple reworking of the proofs of Lemma \ref{l57} and Lemma \ref{l55} shows that as long as $\e$ and 
	$|\rho_0 - \rho_1|$ is sufficiently small, we have for all $\rho$ with  
	$1 < \rho_1 \leq \rho \leq \rho_0$, 
	\begin{align}
	\begin{split}\label{s583}
	\| \pi_\rho^{\perp} \vec u_e(t) \|_{\h(r \geq \rho)} \lesssim (\rho_0 - \rho)^{1/3} \| \pi_\rho \vec u_e(t) \|_{\h(r \geq \rho)} + 
	\| \pi_\rho \vec u_e(t) \|_{\h(r \geq \rho)}^2 + \| \pi_\rho \vec u_e(t) \|_{\h(r \geq \rho)}^3, 
	\end{split}
	\end{align}
	where the implied constant is independent of $\rho$. In the argument, smallness is achieved by taking $\e$ and 
	$|\rho_0 - \rho_1|$ sufficiently small, cutting off the potential term to the exterior region $\{ \rho + t \leq r \leq \rho_0 + t \}$, 
	and using the compact support of $\vec u_e$ along with finite speed of propagation. By taking $\rho_1$ even closer to $\rho_0$ so that
	$|\rho_0 - \rho_1| < \e^3$, we conclude as in Corollary \ref{c510} that
	\begin{align*}
	|\lam_j(\rho_0) - \lam_j(\rho_1)| &\leq C\e \left ( \sum_{i = 1}^{\tilde k} |\lam_i(\rho_1)| +
	\sum_{i = 1}^k |\mu_i(\rho_1)| \right ),  \\
	|\mu_j(\rho_0) - \mu_j(\rho_1)| &\leq C\e \left ( \sum_{i = 1}^{\tilde k}  |\lam_i(\rho_1)| +
	\sum_{i = 1}^k |\mu_i(\rho_1)| \right ).
	\end{align*}
	Since $\lam_j(\rho_0) = \mu_j(\rho_0) = 0$ we conclude by summing the previous expressions that 
	\begin{align*}
	\sum_{i = 1}^{\tilde k} |\lam_i(\rho_1)| +
	\sum_{i = 1}^k |\mu_i(\rho_1)| \leq C(k + \tilde k)\e \left ( \sum_{i = 1}^{\tilde k} |\lam_i(\rho_1)| +
	\sum_{i = 1}^k |\mu_i(\rho_1)| \right ).
	\end{align*}
	By fixing $\e$ sufficiently small, it follows that 
	\begin{align*}
	\sum_{i = 1}^{\tilde k} |\lam_i(\rho_1)| +
	\sum_{i = 1}^k |\mu_i(\rho_1)| = 0.
	\end{align*}
	Thus, $\lam_j(\rho_1) = \mu_j(\rho_1) = 0$.  By Lemma \ref{l54} and \eqref{s583}, we conclude that 
	\begin{align*}
	\| \vec u_e(0) \|_{r \geq \rho_1)} = 0  
	\end{align*}
	which contradicts \eqref{s582}.  Thus, we must have $\rho_0 = \eta$ and $\vec u_e(0,r) = (0,0)$ for $r \geq \eta$ as desired.  
\end{proof}

From the previous argument, we conclude even more for the case $\alpha = 0$. 

\begin{lem}\label{allt lem}
	Let $\alpha$ be as in Lemma \ref{l514}.  If $\alpha = 0$, then 
	\begin{align*}
	\vec u(t,r) = (0,0), \quad \forall (t,r) \in \R \times (0,\infty).
	\end{align*}
\end{lem}

\begin{proof}
	By Lemma \ref{l530} we know that if $\alpha = 0$ then $\vec u(0,r) = (0,0)$ 
	on $\{ r \geq \eta\}$.  
	By finite speed of propagation, we conclude that 
	\begin{align}\label{finite speed}
	\vec u(t,r) = (0,0) \quad \mbox{ on } \{ r \geq |t| + \eta \}. 
	\end{align}
	Let $t_0 \in \R$ be arbitrary and define $u_{t_0}(t,r) = u(t+t_0,r)$.  Then $\vec u_{t_0}$ inherits the following compactness property from $\vec u$: 
	\begin{align*}
	\forall R \geq 0, \quad \lim_{|t| \rar \infty} \| \vec u_{t_0}(t) \|_{\h(r \geq R + |t|; \la r \ra^{d-1} dr )} &=  0, \\
	\lim_{R \rar \infty} \left [ \sup_{t \in \R} \| \vec u_{t_0}(t) \|_{\h(r \geq R + |t|; \la r \ra^{d-1} dr)} \right ] &= 0,
	\end{align*}
	and by \eqref{finite speed} $\vec u_{t_0}(0,r)$ is supported in $\{ 0 < r 
	\leq \eta + |t_0| \}$. By the proof of Lemma \ref{l530} applied to $\vec u_{t_0}$ we conclude that $\vec u_{t_0}(0,r) = (0,0)$ on $r \geq \eta$.  Since $t_0$ was arbitrary, we conclude that
	\begin{align*}
	\vec u(t_0,r) = (0,0) \quad \mbox{on } \{ r \geq \eta \},
	\end{align*}  
	for any $t_0 \in \R$.  Since $\eta > 0$ was arbitrarily fixed in the beginning of this subsection, we conclude that 
	\begin{align*}
	\vec u(t,r) = (0,0), \quad \forall (t,r) \in \R \times (0,\infty).
	\end{align*}
\end{proof}

We now consider the general case for $\alpha$.

\begin{lem}\label{l531}
	Let $\alpha$ be as in Lemma \ref{l514}. As before, we denote the unique $\ell$--equivariant finite 
	energy harmonic map of degree $n$ by $Q$ and recall that there exists a unique 
	$\alpha_{\ell,n} > 0$ such that 
	\begin{align*}
	Q(r) = n\pi - \alpha_{\ell,n} r^{-\ell-1} + O(r^{-\ell-3}) \quad \mbox{ as } r \rar \infty.
	\end{align*}
	Let $Q_{\al - \al_{\ell,n}}$ denote the unique solution to \eqref{sa21} with the property that 
	\begin{align}\label{s91}
	Q_{\alpha - \alpha_{\ell,n}}(r) = n\pi + (\alpha - \alpha_{\ell,n}) r^{-\ell-1} + O(r^{-\ell-3}) \quad \mbox{ as } r \rar \infty.
	\end{align} 
	Note that $Q_{\al - \al_{\ell,n}}$ exists and is unique by 
	Proposition \ref{pa22}.  Define a static solution $U_+$ to \eqref{s31} via
	\begin{align*}
	U_+(r) = \la r \ra^{-\ell} \bigl ( Q_{\alpha - \al_{\ell,n}}(r) - Q(r) \bigr ).
	\end{align*}
	Then 
	\begin{align*}
	\vec u(t,r) = (U_+(r),0), \quad \forall (t,r) \in \R \times (0,\infty).
	\end{align*}
\end{lem}

\begin{proof}
	
	Lemma \ref{l531} follows from the proof for the $\al = 0$ case and a change of variables.  Let $Q_{\alpha - \alpha_{\ell,n}}$ be as in the statement of the lemma.  We define
	\begin{align}
	\begin{split}\label{s92}
	u_{\al}(t,r) &:= u(t,r) - \la r \ra^{-\ell} \left ( Q_{\alpha - \alpha_{\ell,n}}(r) - Q(r)  \right ) \\
	&= u(t,r) - U_+(r)
	\end{split}
	\end{align}
	and observe that $u_{\al}$ solves 
	\begin{align*}
	\p_t^2 u_{\al} - \p_r^2 u_\al - \frac{(d-1)r}{r^2 + 1} \p_r u_\al + V_\al(r) u_\al = N_\al(r,u_\al), 
	\end{align*}
	where the potential $V_\al$ is given by 
	\begin{align}\label{s93}
	V_\al(r) = \ell^2 \la r \ra^{-4} + 2 \la r \ra^{-2} ( \cos 2 Q_{\alpha - \alpha_{\ell,n}} - 1 ), 
	\end{align}
	and $N_\al(r,u) = F_\al(r,u) + G_\al(r,u)$ with 
	\begin{align}
	\begin{split}\label{s94}
	F_\al(r,u) &= \ell(\ell+1) \la r \ra^{-\ell-2} \sin^2 (\la r \ra^\ell u) \sin 2 Q_{\alpha - \alpha_{\ell,n}} , \\
	G_\al(r,u) &= \frac{\ell(\ell+1)}{2} \la r \ra^{-\ell-2} \left [ 2 \la r \ra^\ell u - \sin (2 \la r \ra^\ell u) \right ] \cos 2 Q_{\alpha - \alpha_{\ell,n}} .
	\end{split}
	\end{align}
	By \eqref{s91}, the potential $V_\al$ is smooth and satisfies
	\begin{align*}
	V_\al(r) = \ell^2 \la r \ra^{-4} + O ( \la r \ra^{-2\ell-4} ),
	\end{align*}
	as $r \rar \infty$ and the nonlinearities $F_\al$ and $G_\al$ satisfy 
	\begin{align*}
	|F_\al(r,u)| &\lesssim  \la r \ra^{-3} |u|^2, \\
	|G_\al(r,u)| &\lesssim  \la r \ra^{d-5}|u|^3, 
	\end{align*}
	for $r \geq 0$.  Moreover, by \eqref{s92} we see that $\vec u_{\alpha}$ inherits the compactness property from $\vec u$:
	\begin{align}
	\begin{split}\label{comp prop al}
	\forall R \geq 0, \quad \lim_{|t| \rar \infty} \| \vec u_{\al}(t) \|_{\h( r \geq R + |t|; \la r \ra^{d-1} dr)} = 0, \\
	\lim_{R \rar \infty} \left [ \sup_{t \in \R} \| \vec u_{\al}(t) \|_{\h( r \geq R + |t|; \la r \ra^{d-1} dr)} \right ] = 0.
	\end{split}
	\end{align}
	
	Let $\eta > 0$.  We now define for $r \geq \eta$,
	\begin{align}\label{s96}
	u_{\al,e}(t,r) := \frac{\la r \ra^{(d-1)/2}}{r^{(d-1)/2}} u_{\al}(t,r)
	\end{align}
	and note that $u_{\al,e}$ satisfies an equation analogous to $u_e$:
	\begin{align}\label{s97} 
	\p_t^2 u_{\al,e} - \p^2_r u_{\al,e} - \frac{d-1}{r} \p_r u_{\al,e} + V_{\al,e}(r) u_{\al,e} = N_{\al,e} (r,u_{\al,e}), 
	\quad t \in \R, r \geq \eta,
	\end{align}
	where 
	\begin{align*}
	V_{\al,e}(r) = V_\al(r) - \frac{(d-1)(d-4)}{2 r^2 \la r \ra^2} 
	+ \frac{(d-1)(d-5)}{4 r^{2} \la r \ra^{4}},
	\end{align*}
	and $N_{\al,e}(r,u_e) = F_{\al,e}(r,u_e) + G_{\al,e}(r,u_e)$ where 
	\begin{align*}
	F_{\al,e}(r,u_{\al,e}) &= \frac{\la r \ra^{(d-1)/2}}{r^{(d-1)/2}}  F_\al \left (r, \frac{r^{(d-1)/2}}{\la r \ra^{(d-1)/2}} u_{\al,e} \right ), \\
	G_{\al,e}(r,u_{\al,e}) &= \frac{\la r \ra^{(d-1)/2}}{r^{(d-1)/2}}  G_\al \left (r, \frac{r^{(d-1)/2}}{\la r \ra^{(d-1)/2}}  u_{\al,e} \right ).
	\end{align*}
	In particular, we have the analogues of \eqref{s57}, \eqref{s58}, and \eqref{s59}: for all $r > 0$,
	\begin{align} 
	| V_{\al,e}(r) | &\lesssim r^{-4}, \label{s98} \\
	|F_{\al,e}(r,u)| &\lesssim r^{-3} |u|^{2}, \label{s99} \\
	|G_{\al,e}(r,u)| &\lesssim r^{d-5}|u|^3. \label{s910}
	\end{align} 
	Moreover, $u_{\al,e}$ inherits the following compactness properties from $u_\alpha$:
	\begin{align}
	\begin{split}\label{s911}
	\forall R \geq \eta, \quad \lim_{|t| \rar \infty} \| \vec u_{\al,e}(t) \|_{\h( r \geq R + |t|; r^{d-1} dr)} = 0, \\
	\lim_{R \rar \infty} \left [ \sup_{t \in \R} \| \vec u_{\al,e}(t) \|_{\h( r \geq R + |t|; r^{d-1} dr)} \right ] = 0.
	\end{split}
	\end{align}
	Finally, by construction we see that
	\begin{align}
	\begin{split}\label{e100}
	r^{2-d} u_{\al, e,0}(r) &= O(r^{-2}),  \\
	\int_r^\infty u_{\al,e,1}(\rho) \rho^{2j-1} d\rho &= O(r^{2j - d - 1}), \quad j = 1, \ldots,k. 
	\end{split}
	\end{align}   
	
	Using \eqref{s97}--\eqref{e100}, we may repeat the previous arguments with $u_{e,\al}$ in place
	of $u_e$ to conclude the following analog of Lemma \ref{l530}:
	
	\begin{lem}\label{l530 al}
		$\vec u_{\al}(0,r) = (0,0)$ for $r \geq \eta$. 
	\end{lem}
	
	Finally, we obtain the following analog of Lemma \ref{allt lem}:
	
	\begin{lem}\label{allt lem al}
		We have 
		\begin{align*}
		\vec u_{\alpha}(t,r) = (0,0)
		\end{align*}
		for all $t \in \R$ and $r > 0$. 
	\end{lem}
	
	Equivalently, Lemma \ref{allt lem al} states that 
	\begin{align*}
	\vec u(t,r) = (U_+(r),0) 
	\end{align*}
	for all $t \in \R$ and $r > 0$. This concludes the proof of Lemma \ref{l531} and Proposition \ref{p53}.  
	
\end{proof}

\subsection{Proof of Proposition \ref{static soln}}

Using Proposition \ref{p53} and its analog for $r < 0$, we quickly conclude the proof of Proposition \ref{static soln}.  Indeed, we know that there exists static solutions $U_{\pm}$ to \eqref{s31} such that 
\begin{align}\label{allt pmr}
\vec u(t,r) = (U_{\pm}(r), 0)
\end{align}
for all $\pm r > 0$ and $t \in \R$.  In particular, $\p_t u(t,r) = 0$, 
$\p_r u(t,r) = \p_r u(0,r)$ and $u(t,r) = u(0,r)$ for all $t$ and almost every $r$.  Let $\psi \in C^\infty_0(\R)$ with 
$\int \psi dt = 1$ and let $\varphi \in C^\infty_0(\R)$.  Then since 
$u$ solves \eqref{s31} in the weak sense, we conclude that 
\begin{align*}
0 &= \int \int \bigl [ \psi'(t) \varphi(r) \p_t u(t,r) + \psi (t)\varphi'(r) \p_r u(t,r) + V(r) \psi(t) \varphi(r) u(t,r) \\&\hspace{1.2 cm} - \psi(t) \varphi(r) N(r,u(t,r))  \bigr ] \la r \ra^{d-1} dr dt \\
&= \int \int \psi(t) \bigl [ \varphi'(r) \p_r u(0,r) + V(r) \varphi(r) u(0,r) - \varphi(r) N(r,u(0,r))  \bigr ] \la r \ra^{d-1} dr dt \\
&= \int \bigl [ \varphi'(r) \p_r u(0,r) + V(r) \varphi(r) u(0,r) - \varphi(r) N(r,u(0,r))  \bigr ] \la r \ra^{d-1} dr.
\end{align*}
Since $\varphi$ was arbitrary, we see that $u(0,r)$ is a weak solution
in $H^1(\R)$ to the static equation $-\p_r^2 u - \frac{(d-1)r}{r^2+1} \p_r u + V(r) u = N(r,u)$ on $\R$.  By simple elliptic arguments we conclude that $u(0,r)$ is a classical solution.  Thus, $\vec u(t,r) = (U(r),0) := (u(0,r),0)$ for all $t,r \in \R$ as desired. 

\qed

\subsection{Proofs of Proposition \ref{p51} and Theorem \ref{t21}}

We now briefly summarize the proofs of Proposition \ref{p51} and Theorem \ref{t21}. 

\begin{proof}[Proof of Proposition \ref{p51}]
	By Proposition \ref{static soln}, we have that $\vec u = (U, 0)$ for some finite energy static solution to \eqref{s31}.  Thus, 
	$\psi = Q_{\ell,n} + \la r \ra^\ell u$ is a finite energy solution to \eqref{s21}, i.e. a harmonic map.  
	By the uniqueness part of Proposition \ref{pa21}, we conclude that  
	$\vec u = (0,0)$ as desired.  
\end{proof}

\begin{proof}[Proof of Theorem \ref{t21}]
	Suppose that Theorem \ref{t21} fails.  Then by Proposition \ref{p34}, there exists a nonzero solution $u_*$ to \eqref{s31} such that 
	the trajectory 
	\begin{align*}
	K := \left \{ \vec u_*(t) : t \in \R \right \},
	\end{align*}
	is precompact in $\h(\R; \la r \ra^{d-1} dr)$.  By Proposition \ref{p51} we conclude that $\vec u_* = (0,0)$, which contradicts the fact
	that $u_*$ is nonzero.  Thus, Theorem \ref{t21} must hold.  
\end{proof}

\end{document}